\documentclass[10 pt,a4paper]{article}

\usepackage{amsfonts,amsmath,amsthm,amssymb,graphicx,fancyhdr,makeidx,amscd}
\usepackage{mathrsfs}

\usepackage[T1]{fontenc}
\usepackage[utf8]{inputenc}
\usepackage{authblk}

\usepackage[all]{xy}
\newtheorem{definition}{Definition}[section]
\newtheorem{proposition}{Proposition}[section]
\newtheorem{lemma}{Lemma}[section]
\newtheorem{theorem}{Theorem}[section]
\newtheorem{remark}{Remark}[section]
\newtheorem{corollary}{Corollary}[section]

\newtheorem*{example1}{Example}
\newtheorem*{question}{Question}

\newtheorem*{notation}{Notation}
\numberwithin{equation}{section}
\numberwithin{subsection}{section}
\allowdisplaybreaks

\begin{document}


\newcommand{\riem}{(M^m, \langle \, , \, \rangle)}
\newcommand{\Hess}{\mathrm{Hess}\, }
\newcommand{\hess}{\mathrm{hess}\, }
\newcommand{\cut}{\mathrm{cut}}
\newcommand{\ind}{\mathrm{ind}}
\newcommand{\ess}{\mathrm{ess}}
\newcommand{\longra}{\longrightarrow}
\newcommand{\eps}{\varepsilon}
\newcommand{\ra}{\rightarrow}
\newcommand{\vol}{\mathrm{vol}}
\newcommand{\di}{\mathrm{d}}
\newcommand{\R}{\mathbb R}
\newcommand{\C}{\mathbb C}
\newcommand{\Z}{\mathbb Z}
\newcommand{\N}{\mathbb N}
\newcommand{\HH}{\mathbb H}
\newcommand{\esse}{\mathbb S}
\newcommand{\metric}{\langle \, , \, \rangle}
\newcommand{\lip}{\mathrm{Lip}}
\newcommand{\loc}{\mathrm{loc}}
\newcommand{\diver}{\mathrm{div}}
\newcommand{\disp}{\displaystyle}
\newcommand{\rad}{\mathrm{rad}}
\newcommand{\mmetric}{\langle\langle \, , \, \rangle\rangle}
\newcommand{\sn}{\mathrm{sn}}
\newcommand{\cn}{\mathrm{cn}}
\newcommand{\ink}{\mathrm{in}}
\newcommand{\hol}{\mathrm{H\ddot{o}l}}
\newcommand{\capac}{\mathrm{cap}}
\newcommand{\bmo}{\{b <0\}}
\newcommand{\bmuo}{\{b \le 0\}}
\newcommand{\Fk}{\mathcal{F}_k}
\newcommand{\dist}{\mathrm{dist}}
\newcommand{\gr}{\mathcal{G}}
\newcommand{\grg}{\mathcal{G}^{(g)}}
\newcommand{\Ricc}{\mathrm{Ricc}}
\newcommand{\foc}{\mathrm{foc}}
\newcommand{\F}{\mathcal{F}}
\newcommand{\Cf}{\mathcal{C}_f}
\newcommand{\cutf}{\mathrm{cut}_{f}}
\newcommand{\Cn}{\mathcal{C}_n}
\newcommand{\cutn}{\mathrm{cut}_{n}}
\newcommand{\Ca}{\mathcal{C}_a}
\newcommand{\cuta}{\mathrm{cut}_{a}}
\newcommand{\cutc}{\mathrm{cut}_c}
\newcommand{\cutcf}{\mathrm{cut}_{cf}}
\newcommand{\rk}{\mathrm{rk}}
\newcommand{\crit}{\mathrm{crit}}
\newcommand{\diam}{\mathrm{diam}}
\newcommand{\haus}{\mathcal{H}}
\newcommand{\po}{\mathrm{po}}
\newcommand{\gp}{\mathcal{G}_p}
\newcommand{\FF}{\mathcal{F}}
\newcommand{\He}{\mathbb{H}}
\newcommand{\vp}{\varphi}
\newcommand{\ds}{\displaystyle}
\newcommand{\cL}{\mathcal{L}}
\newcommand{\Sph}{\mathbb{S}}%
\newcommand{\supp}{\mathrm{supp}}%
\newcommand{\gru}{\nabla_0u}
\newcommand{\grho}{\nabla_0\zeta}
\newcommand{\grr}{\nabla_0r}

\newcommand{\essem}{\mathds{S}^m}
\newcommand{\erre}{\mathds{R}}
\newcommand{\errem}{\mathds{R}^m}
\newcommand{\enne}{\mathds{N}}
\newcommand{\acca}{\mathds{H}}
\newcommand{\cvett}{\Gamma(TM)}
\newcommand{\mob}{\mathrm{M\ddot{o}b}}
\newcommand{\mab}{\mathfrak{m\ddot{o}b}}
\newcommand{\TT}{\mathcal{T}}
\newcommand{\HHH}{\mathcal{H}}
\newcommand{\GG}{\mathbb{G}}
\def\beq{\begin{equation}}
\def\eeq{\end{equation}}

\title{\textbf{On the role of gradient terms in coercive \\ quasilinear differential inequalities \\ on Carnot groups}}
\author[1]{Guglielmo Albanese\thanks{guglielmo.albanese@unimi.it}}
\author[2]{Luciano Mari \thanks{mari@mat.ufc.br}}
\author[1,2]{Marco Rigoli\thanks{marco.rigoli@unimi.it}}

\affil[1]{\small Dipartimento di Matematica \emph{Federigo Enriques}, Universit\`{a} degli Studi di Milano, Via~Saldini~50, I-20133, Milano (Italy)}
\affil[2]{Departamento de Matem\'atica, Universidade Federal do Cear\'a, Av. Humberto Monte s/n, Bloco 914, 60455-760 Fortaleza (Brazil)}

\renewcommand\Authands{, and }

\maketitle

\normalsize

\begin{abstract}
In the sub-Riemannian setting of Carnot groups, this work investigates \emph{a priori} estimates and Liouville type theorems for solutions of coercive, quasilinear differential inequalities of the type
$$
\Delta_\GG^\varphi u \ge b(x) f(u) l(|\nabla u|).
$$
Prototype examples of $\Delta_\GG^\varphi$ are the (subelliptic) $p$-Laplacian and the mean curvature operator. The main novelty of the present paper is that we allow a dependence on the gradient $l(t)$ that can vanish both as $t \ra 0^+$ and as $t \ra +\infty$. Our results improve on the recent literature and, by means of suitable counterexamples, we show that the range of parameters in the main theorems are sharp. 
\end{abstract}

\section{Introduction}
The search for \emph{a-priori} estimates and Liouville type  properties for solutions of coercive quasilinear differential inequalities of the type 
\begin{equation}\label{equa_laplaphi}
\Delta^\varphi u \ge b(x)f(u)l(|\nabla u|)
\end{equation}
has captured the interest of researchers and stimulated a great amount of work in recent years, especially in the Euclidean space. In this respect, the purpose of the present paper is to investigate the role played by the gradient term $l(|\nabla u|)$ in the qualitative behaviour of solutions of \eqref{equa_laplaphi}. Our setting is that of Carnot groups, although the investigation could be carried over general Riemannian manifolds. 
To begin with, we need to recall some basic facts about Carnot groups, referring to \cite{BonfLancUguzz} for a thorough exposition. A Carnot group $\GG$ (of step $r$) is a  connected, simply-connected nilpotent Lie group with a graded Lie algebra $\mathcal{G} = \oplus_{i=1}^r V^{m_i}_i$, $m_i = \dim V_i$, $i=1,\dots,r$ generated by $V_1$, that is, $[V_1,V_{i}] = V_{i+1}$, $[V_1,V_{r}]=0$. Each Carnot group is isomorphic to a homogeneous Carnot group, that is, a triple $(\R^n, \circ, \{\delta_R\})$ where $\circ$ is a Lie group structure on $\R^n$, $n = \sum_1^r m_i$, and $\{\delta_R\}$ is a distinguished family of Lie group automorphisms (called dilations) acting as follows: if we write $x \in \R^n$ as $(x^{(1)}, \ldots, x^{(r)}) \in \R^{m_1} \times \ldots \times \R^{m_r}$, then
$$
\delta_R(x) \doteq (R x^{(1)}, R^2 x^{(2)}, \ldots, R^r x^{(r)}).
$$
In what follows, we will identify $\GG$ with the associated homogeneous Carnot group $(\R^n,\circ)$ via the procedure described in \cite{BonfLancUguzz} (Theorem 2.2.18). The collection
$$
\{x^\alpha\}_{\alpha=1}^n \doteq \{x^{(1)}_1, \ldots, x^{(1)}_{m_1}, x^{(2)}_1, \ldots, x^{(2)}_{m_2}, \ldots, x^{(r)}_1, \ldots, x^{(r)}_{m_r}\} 
$$
will then denote the induced coordinate system on $\GG$, and integrations will always be performed with respect to the standard Lebesgue measure $\di x$, which is left-invariant on $\GG$. The integer $Q \doteq \sum_{j=1}^r j m_j$ is called the homogeneous dimension of $\GG$. 
\begin{remark}
\emph{The easiest example of homogeneous Carnot group is the Euclidean space $\R^Q$, $Q \ge 1$, with the Lie group structure $+$. We remark that each Carnot group which is different from $(\R^Q,+)$ has homogeneous dimension $Q \ge 4$, being at least of step $2$. 
}
\end{remark}
For each $1 \le j \le m_1$, let $X_j$ be the left-invariant vector field generated by the coordinate vector field $\partial_j$ at the unit element (the origin of $\R^n$). We can endow the span of $\{X_j\}$, called the first layer, with a Riemannian metric $\metric$ given by declaring $\{X_j\}$ to be orthonormal. Consequently, one has a natural notion of horizontal gradient and divergence: for $u \in C^1(\GG)$ and $Y = y^j X_j$ horizontal (i.e. in the first layer), we set
$$
\nabla_0 u \doteq \sum_{j=1}^{m_1} X_j(u)X_j, \qquad \diver_0 Y = \sum_{j=1}^{m_1} X_j(y^j).
$$
Thus, we can define the canonical sub-Laplacian $\Delta_{\GG}$, that is, the hypoelliptic operator  
\begin{equation}\label{sublapla}
\Delta_{\GG}u = \sum_{j=1}^{m_1} X_jX_j(u). 
\end{equation}
A key fact is that $\Delta_\GG$ possesses a fundamental solution $\Gamma(x)$ with pole at the origin. Hence, if $Q \ge 3$ (that is, if $\GG \neq \R,\R^2$), we can consider
$$
r(x) \doteq \Gamma(x)^{\frac{1}{2-Q}},
$$
which is continuous on $\GG$, smooth out of the origin and gives rise to a symmetric, homogeneous norm, that is, $r(x)>0$ iff $x \neq 0$, $r(\delta_R(x))= R r(x)$ and $r(x^{-1})=r(x)$. Being $r$ homogeneous of degree $1$ with respect to dilations, $|\nabla_0 r|$ is homogeneous of degree $0$, whence in particular $|\nabla_0r|$ is bounded. It can be proved that each pair of symmetric homogeneous norms $r_1,r_2$ on $\GG$ are equivalent, that is, they satisfy $C^{-1}r_1 \le r_2 \le Cr_1$ on $\GG$ for some constant $C>0$; thus, for the purpose of our paper, it is irrelevant which norm we will use, and, for this reason, hereafter we will fix one such $r$ and rescale it in order to satisfy $|\nabla_0r| \le 1$. As for the cases $Q=1$ and $Q=2$, that is, when $\GG = \R$ or $\GG=\R^2$, in what follows we will consider the Euclidean homogeneous norm $r(x) = |x|$. Let $B_R \doteq \{ x : r(x) < R\}$ be a sub-level set for $r$, hereafter called a ball of radius $R$. The homogeneity of $r$ with respect to $\delta_R$ imply that the volume of $B_R$ grows polynomially in $R$: indeed, changing variables according to $x = \delta_R(y)$, we have $\det((\delta_R)_*) = R^Q$ and thus 
\begin{equation}\label{volgrowth}
\vol\big(B_R\big) = \int_{\{r(x) <R\}} \di x = R^Q \int_{\{r(y) <1\}} \di y = CR^Q.
\end{equation}

\begin{example1}
\emph{The simplest, non-trivial example of Carnot group is the Heisenberg group $\He^m$ of real dimension $2m+1$, that is, the manifold $\He^m=\C^m \times \R$ with group structure given by 
$$
(z,t) \circ (z',t')=(z+z',t+t'+2 \mathrm{Im} \langle z,z'\rangle), \qquad \forall \, (z,t),(z',t') \in \He^m,
$$
where $\metric$ denotes the usual hermitian product in $\C^m$. A basis for the left-invariant vector fields  of $\He^m$ is given by
\begin{equation}\label{base}
X_k=2 \mathrm{Re} \frac{\partial}{\partial z_k}+2 \mathrm{Im} z_k	\frac{\partial}{\partial t} \ , \quad 
Y_k=2 \mathrm{Im} \frac{\partial}{\partial z_k}-2 \mathrm{Re} z_k	\frac{\partial}{\partial t} \ , \quad \frac{\partial}{\partial t}
\end{equation}
for $k=1,\ldots, m$, and they satisfy Heisenberg's canonical commutation relations
\begin{equation}\label{comm}
\qquad [X_j,Y_k]=-4\delta_{jk}\frac{\partial}{\partial t},
\end{equation}
all the other commutators being zero. The fields $\{X_j, Y_j\}_{j=1}^m$ generate the first layer, and dilations are given by $\delta_R(z,t) = (R z, R^2 t)$. The standard homogeneous norm is 
\begin{equation}\label{dist}
r(x)=r(z,t)=\left( \left| z \right|^4 +t^2 \right)^{\frac{1}{4}},
\end{equation} 
where $|\cdot |$ is the norm in $\C^m$, and a simple computation shows that $|\nabla_0r|^2= |z|^2/r^2 \le 1$.
}
\end{example1}

Given a Carnot group $\GG$, the sub-Laplacian $\Delta_\GG$ can always be written in divergence form with respect to the underlying Euclidean structure: more precisely, If $X_j = c_{j\alpha} \partial_\alpha$ is the expression of $X_j$ in the canonical basis $\{\partial_\alpha\}_{\alpha=1}^n$, then 
\begin{equation}\label{subla_divergence}
\Delta_\GG u = c_{j\alpha}\partial_\alpha\big(c_{j\beta}\partial_\beta u\big) = \partial_\alpha \big(c_{j\alpha}c_{j\beta}\partial_\beta u\big), 
\end{equation}
where the last identity follows from an important property of Carnot groups, namely, the fact that the coefficient $c_{j\alpha}$ does not depend on the $\alpha$-th coordinate (\cite{BonfLancUguzz}, p. 64). Consequently, we can consider weak solutions of differential inequalities with $\Delta_\GG$ in regularity classes less demanding than $C^2(\GG)$. More generally, given  $\varphi$ satisfying
\begin{equation}\label{ipo_varphi}
\varphi \in C^0\big([0,+\infty)\big), \qquad \varphi \ge 0 \ \text{ on } \, [0,+\infty),
\end{equation}
one can define a $\varphi$-Laplacian by setting 
$$
\Delta_{\GG}^\varphi u = \diver_0\left( \frac{\varphi(|\nabla_0u|)}{|\nabla_0u|} \nabla_0u\right)
$$
in the weak sense. Typical examples are:
\begin{itemize}
\item[-] the $p$-Laplacian, where $\varphi(t) = t^{p-1}$ and $p \ge 1$;
\item[-] the (generalized) mean curvature operator, where $\varphi(t) = t^{k-1}(1+ t^k)^{-1/2}$ and $k \ge 1$. The usual mean curvature operator is recovered for $k=2$.
\end{itemize} 
In what follows, for convenience we will just consider $\lip_\loc$-solutions. We recall that a function $u : U \subset \GG \ra \R$ is said to be Lipschitz if there exists $C>0$ such that
\begin{equation}\label{lipsch}
u(x)-u(y) \le C r(x^{-1}y) \qquad \text{for each $x,y \in U$,}
\end{equation}
where $r$ is a fixed homogeneous norm. By \eqref{subla_divergence} and a standard argument (see for instance \cite{evansgariepy}, Thm.5 p. 137), it is not hard to show that \eqref{lipsch} implies the existence and the local boundedness of the weak derivatives $X_j(u)$. 
\begin{notation}
\emph{Hereafter, given a function $u$ on $\GG$, we denote with $u^* \doteq \sup_\GG u$, $u_* \doteq \inf_\GG u$. We set $\R^+_0\doteq [0,+\infty)$, $\R^+ \doteq (0,+\infty)$.
}
\end{notation}

\begin{definition}\label{def_weakform}
Let $f\in C^0(\R)$, $l \in C^0(\R_0^+)$ and $b(x) \in C^0(\GG)$. A function $u \in \lip_\loc(\GG)$ is said to be a weak solution of 
\begin{equation}\label{ineq_usiamo}
\Delta^\varphi_{\GG}u \ge b(x)f(u)l(|\nabla_0u|) \qquad \text{(resp. $\le$)}
\end{equation}
if, for each $0 \le \phi \in \lip_c(\GG)$, 
\begin{equation}\label{def_weakequa}
\int_{\GG} \frac{\varphi(|\nabla_0u|)}{|\nabla_0u|} \langle \nabla_0 u, \nabla_0 \phi \rangle \di x \le - \int_{\GG} b(x)f(u)l(|\nabla_0u|) \phi \di x \qquad \text{(resp. $\ge$).}
\end{equation}
\end{definition}
\vspace{0.1cm}
In the present paper, we investigate solutions $u \in \lip_\loc(\GG)$ of 
\begin{equation}\label{equa_padraoHen}
\Delta^\varphi_{\GG} u \ge b(x)f(u) l(|\nabla_0u|), 
\end{equation}
possibly changing sign, under the following basic assumptions:
\begin{equation}\label{assumptions}
\left\{\begin{array}{ll}
\text{(WpC)} & : \quad \left\{\begin{array}{ll} \varphi \in C^0(\R^+_0), \quad \varphi(0)=0, \quad \varphi(t) > 0 \ \text{ on } \, \R^+; \\[0.1cm]
\text{there exists $p \ge 1$, $C>0$ such that } \, \varphi(t) \le Ct^{p-1} \ \text{ on } \, \R^+;
\end{array}\right. \\[0.5cm]
 & \quad b \in C^0(\GG), \qquad b > 0 \ \text{ on } \, \GG; \\[0.2cm]
& \quad f \in C^0(\R), \qquad l \in C^0(\R_0^+), \qquad l(t)>0 \ \text{if } \, t>0.
\end{array}\right.
\end{equation}

Agreeing with the notation in \cite{DAmbrMit, dambrosiomitidieri_2}, the tag (WpC) above denotes the \emph{weak $p$-coercivity} of $\Delta_\GG^\varphi$. Note that the $p$-Laplacian is (WpC), and the generalized mean curvature operator is (WpC) for each $k/2\le p \le k$. \par

\begin{remark}
\emph{We have decided to consider locally Lipschitz functions for the sake of simplicity, but our theorems could be stated for Sobolev classes of solutions. This setting is more appropriate for problem \eqref{equa_padraoHen} since, as remarked in \cite{dambrosiomitidieri_2}, in the generality \eqref{assumptions} a weak Harnack inequality for solutions of \eqref{equa_padraoHen} seems still missing, hence $u$ is not even guaranteed to be locally bounded. However, investigating \eqref{equa_padraoHen} without the property $u \in L^\infty_\loc(\GG)$ requires various non-trivial adjustments (see \cite{dambrosiomitidieri_2} for details) which we prefer to avoid in order to help readability.  
}
\end{remark}

We are interested in the following two problems:
\begin{itemize}
\item[(P1)] under which conditions relating $\varphi,f,l,b$ we can obtain a maximum principle at infinity, stating that \emph{a-priori} slowly growing solutions of \eqref{equa_padraoHen} are bounded above and satisfy $f(u^*) \le 0$ (if they are not constant);
\item[(P2)] which is the optimal growth condition on $f$ (in the spirit of Keller and Osserman's  works \cite{keller,osserman}), in terms of $l,\varphi,b$, to guarantee the following \emph{a-priori} estimate: each solution of \eqref{equa_padraoHen} is bounded above.
\end{itemize}
From solving (P1) and (P2), one can readily obtain Liouville type  theorems for particular $f$ and $u$ solving \eqref{equa_padraoHen}. For instance, given $\varphi,b,f,l$, if both the properties listed in (P1),(P2) hold and $f>0$ on $\R^+$, then there are no non-constant, non-negative solutions of \eqref{equa_padraoHen}: this because (P1) and (P2) would imply that any non-constant solution $u$ of \eqref{equa_padraoHen} must be bounded and satisfy $f(u^*) \le 0$, hence $u^* \le 0$.\par

Our approach to (P1) has its roots in the works \cite{rigolisalvatorivignati, rigolisalvatorivignati_2, rigolisalvatorivignati_3, prs_gafa} by the third author and his collaborators, and in the subsequent improvements in \cite{PRS_MEM, mrs}. Interesting Liouville theorems for slowly growing solutions have also been shown in \cite{farinaserrin, pucciserrin_2, dambrosiomitidieri_2} for a broad class of differential inequalities including \eqref{equa_padraoHen}. However, as we shall see, the results in \cite{farinaserrin, pucciserrin_2, dambrosiomitidieri_2} are skew with our Theorem \ref{teo_main_2} below, that is, the range of parameters considered is quite different from our. Regarding problem (P2), it has recently generated a vast literature in the present setting of Carnot groups, see \cite{MMMR, brandomagliaro, bordofilipucci, DAmbrMit, dambrosiomitidieri_2}, and will be commented in awhile. However, it seems to us that the role of the term $l(|\nabla_0u|)$ when dealing with problems (P1) and (P2) still needs to be clarified, especially in the case when 
\begin{equation}\label{assumpt_l0infty}
\lim_{t \ra 0^+} l(t) = 0, \qquad \lim_{t \ra +\infty} l(t)=0.
\end{equation}
Evidently, assumption \eqref{assumpt_l0infty} is the gradient dependence that makes the validity of the properties in (P1) and (P2) more difficult to achieve. What are the optimal decay rates of $l$ at zero and infinity that guarantee the solvability of (P1) and (P2)? How is this decay related to the behaviour of $f,b$ and to that of $\varphi$?\par 
The outcome of the present paper are two sharp criteria, Theorems \ref{teo_main_2} and \ref{teo_main} below, to answer problems (P1) and (P2). We deal with a large family of $\varphi$-Laplace operators with special emphasis on mean curvature type ones, and we investigate gradient dependences $l(t)$ that may vanish both at zero and at infinity. \par 

Before stating our main theorems, we give a brief account on related works. The literature on \eqref{equa_padraoHen} is huge, and for this reason we decided to focus on those results specifically for problems (P1) and (P2) that allow a non-trivial gradient dependence $l(t)$. \par
To the best of our knowledge, Liouville type   theorems for global solutions of \eqref{equa_padraoHen} have mainly been investigated by means of two different approaches: the first rests on radialization techniques and refined comparison theorems (\cite{mrs, MMMR, brandomagliaro, bordofilipucci, pucciserrin_2}), while the second  is directly based on the weak formulation, via a careful choice of test functions, in the spirit of the work of Mitidieri-Pohozaev \cite{MitidPohoz} (\cite{DAmbrMit, farinaserrin, dambrosiomitidieri_2}). Radialization techniques exploit the properties of a homogeneous norm $r(x)$ in order to construct suitable radial supersolutions. In doing so, an integral condition relating $f,l, \varphi$ naturally appears, that is sufficient and in many cases necessary for the existence of uniform estimates from above for solutions of \eqref{equa_padraoHen}. We briefly recall how this condition, called the Keller-Osserman condition, is defined. Suppose that $f>0$ on $\R^+$, and set 
\begin{equation}\label{def_Fe}
F(t) \doteq \int_0^t f(s) \di s.
\end{equation}
Assuming that 
\begin{equation}\label{phiel}
\varphi \in C^1(\R^+), \qquad \varphi'>0 \ \text{ on } \, \R^+, \qquad \frac{s\varphi'(s)}{l(s)} \in L^1(0^+)\backslash L^1(+\infty), 
\end{equation}
the function 
\begin{equation}\label{def_K}
K(t) \doteq \int_0^t \frac{s\varphi'(s)}{l(s)}\di s
\end{equation}
realizes a homeomorphism of $\R^+_0$ onto itself, and thus it admits an inverse $K^{-1} : \R^+_0 \ra \R^+_0$. The Keller-Osserman condition for \eqref{equa_padraoHen} is the next integrability requirement:
\begin{equation}\label{KO}\tag{KO}
\frac{1}{K^{-1}\circ F} \in L^1(+\infty).
\end{equation}
When $\varphi(t) = t^{p-1}$, which is the case of the $p$-Laplacian, and for $l \equiv 1$, \eqref{KO} takes the well-known expression
$$
\frac{1}{F^{1/p}} \in L^1(+\infty).
$$
It is important to note that \eqref{KO} does \emph{not} depend on the underlying (Riemannian or sub-Riemannian) space. However, geometric data such as curvatures and volume growth of balls  appear as restrictions on $l,\varphi$ to ensure that \eqref{KO} implies uniform estimates from above (see \cite{mrs}). The origin of this restriction is deep and not yet clarified, and is the subject of a forthcoming paper \cite{BMPR}.\par
The Keller-Osserman condition \eqref{KO} originated from the papers \cite{keller, osserman} in the prototype case $\Delta u \ge f(u)$ on $\R^m$ and, as far as we know, first appeared for nontrivial $l(t)$ in a paper of R. Redheffer \cite{redheffer} (Corollary 1 therein) in the investigation of the inequality $\Delta u \ge f(u)l(|\nabla u|)$. Since then, it has been systematically studied by various authors, and for \eqref{equa_padraoHen} with nontrivial $l(t)$ we stress \cite{caristimiti} (for the 1-dimensional problem), \cite{bandlegrecoporru, greco, filippucci} (when $\Delta^\varphi_\GG$ is the mean curvature operator) and \cite{MartioPorru, filipupuccirigoli, filippucci} (when $\Delta^\varphi_\GG$ is the $p$-Laplacian).
\par

Regarding the sub-Riemannian setting, in Theorem 1.1 of \cite{MMMR} the authors showed that \eqref{equa_padraoHen} with $b(x) \equiv 1$ has no non-negative solutions on $\GG=\He^n$ if \eqref{KO} holds, provided that $\varphi, l$ satisfy some homogeneity conditions subsequently removed in \cite{brandomagliaro}. Companion existence results under the failure of \eqref{KO} have been given in \cite{MMMR} (for the $p$-Laplacian on $\He^n$), in \cite{brandomagliaro} (for the $\varphi$-Laplacian on a Carnot group), and in the very recent \cite{bordofilipucci} on $\He^n$. We underline that the extension from $\He^n$ to a general Carnot group is not merely ``cosmetic", since, in this more general setting, radialization techniques yield ordinary differential equations difficult to tackle because of the existence of an extra term, namely, $\langle \nabla_0|\nabla_0 r|^2, \nabla_0 r \rangle$, which is identically zero for the standard homogeneous norm in  $\He^n$. Note that the equality $\langle \nabla_0|\nabla_0 r|^2, \nabla_0 r \rangle=0$ means that the norm is $\infty$-harmonic.\par 
On the positive side, \eqref{KO} is sharp and mildly demanding on $f$ and $l$; in particular it does not require $f$ or $l$ to be of polynomial type. Furthermore, it can be considered for a wide class of quasilinear operators of geometric and analytic interest. However, in this setting there are two drawbacks: firstly, the results in \cite{MMMR, brandomagliaro, bordofilipucci} require a (relaxed) monotonicity of $l$, namely the first two need $f$ and $l$ to be $C$-increasing, that is, 
$$
\sup_{[0,t]} l \le Cl(t) \qquad \text{on $\R^+_0$,} \ \text{ for some constant } \, C \ge 1,
$$
while the third needs $l$ to be $C$-decreasing. In particular, none of these results admits the possibility that $l$ vanishes at the same time at zero and at infinity\footnote{In this respect, there is an inaccuracy in Theorem 1.1 of \cite{MMMR}: the assumptions should have included condition $l(0)>0$, which is necessary for the proof to work in $\He^n$; however, $l(0)>0$ can be removed with a careful measure-theoretic argument, as shown in \cite{brandomagliaro_2}.}. Secondly, their technique heavily uses the continuity of the (horizontal) gradient of the solutions, and it seems hard to adapt the argument to solutions with a regularity weaker than $C^1$. Another feature of the above method is that, when $l$ is bounded from below by a positive constant at infinity, the second condition in \eqref{phiel} is not taylored for the mean curvature operator. Operators of this kind, more precisely those with $\varphi$ satisfying $t \varphi'(t) \le C\varphi(t)$ on $\R^+$, have been considered in Section 4 of \cite{mrs} in a Riemannian setting (see also \cite{greco}). There, the authors propose to replace $s\varphi'(s)$ with $\varphi(s)$ in the definition of $K(t)$ in \eqref{def_K}, and in this way they still achieve the non-existence of non-constant, non-negative solutions of \eqref{equa_padraoHen} under the validity of \eqref{KO}. As we will show later, this replacement still gives a sharp result, see the discussion after Remark \ref{rem_belluz} below. It should be noted that Corollary A2 in \cite{mrs} seems to be the first Liouville type  theorem for inequalities of type \eqref{equa_laplaphi} on general Riemannian manifolds, in particular, not requiring a polynomial volume growth of geodesic balls. For the sake of comparison with our main result, we state Corollary A2 in \cite{mrs} in the Euclidean space\footnote{In Corollary A2 in \cite{mrs}, the Euclidean case is achieved by setting $\beta=-2$. In this respect, note that there is typo, a missing minus sign in the right-hand side of inequality (1.12) therein, which should be replaced by 
$$
\Ricc_{n,m}(L_D) \ge - H^2\big(1 +r^2\big)^{\frac{\beta}{2}}.
$$
} and for the mean curvature operator. 

\begin{theorem}[\cite{mrs}, Cor. A2]\label{teo_JFA}
Let $\mu, \chi  \in \R$ satisfy 
\begin{equation}\label{range_JFA}
0 \le \chi < 1, \qquad 0 \le \mu < 1-\chi, 
\end{equation}
and let $f \in C^0(\R)$, $f(0)=0$, $f>0$ on $\R^+$. Suppose that $f$ is $C$-increasing on $\R^+$ and, setting $F$ as in \eqref{def_Fe}, assume that 
\begin{equation}\label{grande_F}
F^{-\frac{1}{1-\chi}} \in L^1(+\infty).
\end{equation}
Then, each $C^1$-solution of 
\begin{equation}\label{eq_curvame}
\diver\left( \frac{\nabla u}{\sqrt{1+|\nabla u|^2}}\right) \ge \left( 1+ |x|\right)^{-\mu}f(u)|\nabla u|^\chi \qquad \text{on } \, \R^n
\end{equation}
is either non-positive or constant\footnote{Indeed, the statement of Corollary A2 in \cite{mrs} is that any non-negative solution of \eqref{eq_curvame} is constant. However, Corollary A2 is just a particular case of Theorem A therein, whose conclusion is the one reported in Theorem \ref{teo_JFA} above.}.
\end{theorem}

\begin{remark}\label{rem_strange!}
\emph{Since $f$ is $C$-increasing, $f(t) \ge c_1$ for some constant $c_1>0$ and $t$ large, and integrating we deduce $F(t) \ge c_2t$ for large $t$ and some $c_2>0$. Therefore, if $\chi>0$ assumption \eqref{grande_F} is \emph{automatically satisfied}. Thus, in the particular case $f(t) \ge Ct^\omega$ for $t$ large and $\omega\ge 0$, \eqref{grande_F} is always met if $\chi>0$, and if $\chi=0$ it is satisfied when $\omega>0$. In other words, \eqref{grande_F} is met if $\omega>-\chi$.  We will come back to this point later.
}
\end{remark}

\begin{remark}\label{rem_fili}
\emph{For $f(t) = t^\omega$ and in the range 
\begin{equation}\label{range_filippucci}
0 \le \chi < 1, \qquad \mu < 2-\chi, \qquad \omega > 1-\chi, 
\end{equation}
if $p \in (1,n)$ the above result has been obtained in \cite{filippucci} (see Corollary 1.4 therein; the bound $p \in (1,n)$ is assumed at p.2904).
}
\end{remark}
The technique that we present here is different from the one described above, and is closer, in spirit, to the one developed in \cite{MitidPohoz, DAmbrMit, farinaserrin, farinaserrin_I, dambrosiomitidieri_2}. There, the authors consider inequalities of the type
\begin{equation}\label{lageneral}
\diver \mathcal{A}(x,u,\nabla u) \ge \mathcal{B}(x,u,\nabla u)
\end{equation}
and their non-coercive counterparts, on spaces including Carnot groups, under weak ellipticity requirements on $\mathcal{A}$ that, when rephrased for the operator $\Delta_\GG^\varphi$, give (WpC) in \eqref{assumptions}\footnote{To be precise, rephrasing condition (WpC) in \cite{DAmbrMit, dambrosiomitidieri_2} to our setting would give $\varphi \ge 0$ on $\R^+$, instead of the stronger $\varphi>0$ on $\R^+$ which we assume throughout this paper.}. As far as we know, the only result that allows the gradient term to vanish both at zero and at infinity is the following, due to A. Farina and J. Serrin \cite{farinaserrin}. In order to compare with our main theorems, we rename their parameters to agree with our notation, and we state their theorem for
$$
\mathcal{A}(x,z,\rho) = \frac{\varphi(|\rho|)}{|\rho|}\rho, \qquad \mathcal{B}(x,z,\rho) = b(x)f(z)l(|\rho|).
$$


\begin{theorem}[\cite{farinaserrin}, Thm 6.]\label{teo_farinaserrin}
On the Euclidean space $\GG=\R^n$, consider $\varphi,b,f,l$ satisfying assumptions \eqref{assumptions} for some $p>1$.
Assume  that
\begin{equation}\label{ipo_bfarise}
b(x) \ge C_1\big( 1+ |x|\big)^{-\mu} \qquad \text{on } \, \R^n,
\end{equation}
for some $\mu \in \R$, $C_1 \in \R^+$, and that
\begin{equation}\label{condifl_farinaserrin}
\begin{array}{l}
\disp tf(t) \ge C_2 \cdot\left\{ \begin{array}{lll}
|t|^{\omega_1+1} & \quad \text{if } \, |t| \le 1, & \quad \text{for some } \, \omega_1 \ge 0\\[0.1cm]
|t|^{\omega_2+1} & \quad \text{if } \, |t| \ge 1, & \quad \text{for some } \, \omega_2 \in \R \\[0.1cm]
\end{array}\right. \\[0.7cm]
\disp l(t) \ge C_2 \cdot \left\{ \begin{array}{lll}
t^{\chi_1} & \quad \text{if } \, t \in [0,1], & \quad \text{for some } \, \chi_1 \ge 0 \\[0.1cm]
t^{\chi_2} & \quad \text{if } \, t > 1, & \quad \text{for some } \, \chi_2 \in \R, \\[0.1cm]
\end{array}\right.
\end{array}
\end{equation}
for some $C_2 \in \R^+$. If $l(0)>0$, suppose further that $\chi_1=0$. Set 
\begin{equation}\label{parafarina}
\underline{\omega} \doteq \min\{\omega_1,\omega_2\}, \quad  
\underline\chi \doteq \min\{\chi_1,\chi_2\}, \quad \bar \chi \doteq \max\{\chi_1,\chi_2\} \ge 0.
\end{equation}
If 
\begin{equation}\label{condifarina}
\bar\chi \le p-1, \qquad \mu < p- \bar\chi, \qquad \underline{\omega} > p-1-\underline{\chi}, 
\end{equation}
then any $C^1$-solution of 
$$
\diver\left( \frac{\varphi(|\nabla u|)}{|\nabla u|}\nabla u \right) = b(x) f(u)l(|\nabla u|) \qquad \text{on } \, \R^n
$$
must be constant.
\end{theorem}

\begin{remark}
\emph{The first condition in \eqref{condifl_farinaserrin} implies that $t=0$ is the unique zero of $f$\footnote{For the validity of Theorem 6 in \cite{farinaserrin}, the authors assume that $\mathcal{B}(x,z,0)=0$ for each $(x,z) \in \R^n\times \R$ (condition (3) therein), which in our setting is granted when $l(0)=0$. However, if $l(0)>0$, the condition turns out to imply $f \equiv 0$, a highly demanding assumption. However, when $l(0)>0$ and $\chi_1=0$, according to the remark at the end of p.4410 in \cite{farinaserrin} it is sufficient that $\mathcal{B}(x,0,0)=0$, which is automatic since $f(0)=0$. This is the reason why we require $\chi_1=0$ when $l(0)>0$.}. 
}
\end{remark}

As it is apparent in \eqref{condifarina}, a fast decay of $l$ at infinity (that is, a highly negative $\underline{\chi}$) forces $\underline{\omega}$ (hence $\omega_1$ and $\omega_2$) to be very large in order to ensure the Liouville property. If we agree to consider a reasonable Keller-Osserman condition to be \eqref{KO} with the modified choice
\begin{equation}\label{modi_KO}
K(t) = \int_0^t \frac{\varphi(s)}{l(s)} \di s,
\end{equation}
and we suppose that $\varphi(t) \approx t^{p-1}$ at infinity, then we see that the third in \eqref{condifarina} is a sufficient condition for \eqref{KO} to hold.\\

We are ready to state our main results. The first guarantees \emph{a-priori} estimates for solutions $u$ of \eqref{equa_padraoHen} when $f$ enjoys a Keller-Osserman type condition.

\begin{theorem}\label{teo_main}
Let $\GG$ be a Carnot group with homogeneous norm $r$, and consider $\varphi, b,f,l$ meeting the assumptions in  \eqref{assumptions} for some $p>1$. Assume that, for some $\mu,\chi, \omega \in \R$ satisfying  
\begin{equation}\label{pararange}
0 \le \chi \le p-1, \qquad \mu< p-\chi, \qquad \omega > p-1 - \chi,
\end{equation}
the following inequalities hold:
\begin{equation}\label{assu_main}
\begin{array}{ll}
b \ge C(1+r)^{-\mu} & \quad \text{on } \, \GG, \\[0.2cm]
f(t) \ge C t^\omega & \quad \text{for } \, t \, \text{ large enough,} \\[0.2cm]
\disp l(t) \ge C\frac{\varphi(t)}{t^{p-1-\chi}} & \quad \text{on } \, \R^+,
\end{array}
\end{equation}
for some constant $C>0$. Then, for each solution $u \in \lip_\loc(\GG)$ of 
\begin{equation}\label{main_case}
\Delta_\GG^\varphi u \ge b(x)f(u)l(|\nabla_0u|) \qquad \text{on } \, \GG,
\end{equation}
it holds $u^* < +\infty$ and 
\begin{itemize}
\item[-] if $l(0)>0$, then $f(u^*) \le 0$;
\item[-] if $l(0)=0$, then $f(u^*) \le 0$ unless $u$ is constant.
\end{itemize}
Moreover, if the second in \eqref{assu_main} is replaced by 
\begin{equation}\label{equazione_fmain_simme}
tf(t) \ge C|t|^{\omega+1} \qquad \text{for } \, \ |t| \, \text{ large enough}, 
\end{equation}
then each solution $u$ of \eqref{main_case} with the equality sign satisfies $u \in L^\infty(\GG)$ and,  
\begin{itemize}
\item[-] if $l(0)>0$, then $f(u^*) \le 0 \le f(u_*)$;
\item[-] if $l(0)=0$, then $f(u^*) \le 0 \le f(u_*)$ unless $u$ is constant.
\end{itemize}
\end{theorem}

\begin{remark}\label{rem_constants}
\emph{When $l(0)=0$, note that each constant function satisfies \eqref{main_case} with the equality sign. On the other hand, if $l(0)>0$, a constant $u=c$ solves \eqref{main_case} (respectively, with the equality sign) if and only if $f(c) \le 0$ (resp. $f(c)=0$). 
}
\end{remark}

\begin{remark}
\emph{In the particular case $l(t)\equiv 1$ and $\chi=0$, Theorem \ref{teo_main} has been proved in \cite{dambrosiomitidieri_2} (Corollary 8.6 therein) and \cite{farinaserrin_I} (Theorems 1 and 2 therein). We observe that Theorems 1 and 2 in \cite{farinaserrin_I} also consider the case $\mu > p$.  
}
\end{remark}


Our second main theorem deals with slowly growing solutions. To the best of our knowledge, this is the first result in the literature for the range of parameters \eqref{pararange_2} below, see also the next Remarks \ref{rem_liou} and \ref{rem_farinaserrinliou}. In view of Remark \ref{rem_constants}, we just concentrate on non-constant solutions. We underline that no restriction on the homogeneous dimension of $\GG$ is needed.

\begin{theorem}\label{teo_main_2}
Let $\GG$ be a Carnot group with homogeneous norm $r$, and consider $\varphi, b,f,l$ meeting the assumptions in  \eqref{assumptions} for some $p>1$. Assume that, for some $\mu,\chi \in \R$ satisfying  
\begin{equation}\label{pararange_2}
0 \le \chi < p-1, \qquad \mu < p-\chi,
\end{equation}
the following inequalities hold:
\begin{equation}\label{assum_main_2}
\begin{array}{ll}
b \ge C(1+r)^{-\mu} & \quad \text{on } \, \GG, \\[0.2cm]
f(t) \ge C & \quad \text{for } \, t \, \text{ large enough,} \\[0.2cm]
\disp l(t) \ge C\frac{\varphi(t)}{t^{p-1-\chi}} & \quad \text{on } \, \R^+,
\end{array}
\end{equation}
for some constant $C>0$. Let $u \in \lip_\loc(\GG)$ be a non-constant, weak solution of 
\begin{equation}\label{main_case_2}
\Delta_\GG^\varphi u \ge b(x)f(u)l(|\nabla_0u|) \qquad \text{on } \, \GG
\end{equation}
such that
\begin{equation}\label{opequeno}
u_+(x) = o\left(r(x)^{\frac{p-\chi-\mu}{p-\chi-1}}\right) \qquad \text{as } \, r(x) \ra +\infty.
\end{equation}
Then, $u$ is bounded above and $f(u^*) \le 0$.\\
If $u$ solves \eqref{main_case_2} with the equality sign, 
\begin{equation}\label{assu_barf}
tf(t) \ge C|t| \qquad \text{for } \, \ |t| \, \text{ large enough,} 
\end{equation}
and \eqref{opequeno} is valid with $u$ instead of $u_+$, then under the above assumptions $u \in L^ \infty(\GG)$ and $f(u^*) \le 0 \le f(u_*)$.
\end{theorem}

\begin{remark}
\emph{We will show that the range of parameters in \eqref{pararange}, \eqref{pararange_2} is sharp, as well as the growth condition \eqref{opequeno}. 
}
\end{remark}

\begin{remark}\label{rem_othercases}
\emph{If $\mu = p-\chi$, then the conclusions of Theorem \ref{teo_main_2} still hold if \eqref{opequeno} is replaced by the weaker requirement that $u$ be bounded above. On the other hand, suppose that $p>Q$, $Q$ being the homogeneous dimension of $\GG$. If $\mu< p-\chi$ and \eqref{opequeno} is not satisfied, but still 
\begin{equation}\label{ogrande}
u(x) = \mathcal{O}\left(r(x)^{\frac{p-\chi-\mu}{p-\chi-1}}\right) \qquad \text{as } \, r(x) \ra +\infty,
\end{equation}
then the conclusions hold provided that
\begin{equation}\label{bound_border4}
\frac{p-\chi-\mu}{p-\chi-1} \le \frac{p-Q}{p-1}.
\end{equation}
} 
\end{remark}

\begin{remark}\label{rem_liou}
\emph{In view of \eqref{ogrande} and \eqref{bound_border4}, it is interesting to compare Theorem \ref{teo_main_2} with Theorem 1.1 in  \cite{pucciserrin_2} (see also Theorem 10 in \cite{farinaserrin}). There, the authors obtain the constancy of solutions of \eqref{main_case_2} on $\R^Q$ (with the equality sign, and with $tf(t) \ge 0$ on $\R$ \footnote{It should be observed that assumption (1.3) in \cite{pucciserrin_2}, when rephrased for the inequality \eqref{main_case_2}, gives necessarily $\varphi(t) = Ct^{p-2}$. However, the above restriction does not appear in Theorem 10 of \cite{farinaserrin}, which considers the case $f(t) \equiv 0$.}) whenever $p> Q$ and 
\begin{equation}\label{opiccolo}
u(x) = o \left( r(x)^\frac{p-Q}{p-1}\right) \qquad \text{as } \, r(x) \ra +\infty.
\end{equation}
Note that condition \eqref{opiccolo} is sharp and related to the growth of the fundamental solution for the $p$-Laplacian (see \cite{pucciserrin_2} and Remark 10.3 in \cite{dambrosiomitidieri_2}). For $p \ge Q$, further interesting results can be found in \cite{dambrosiomitidieri_2} (Theorems 10.1 and  10.4 therein), where \eqref{opiccolo} is replaced with an asymptotic integral estimate. However, outside of the borderline case described in Remark \ref{rem_othercases}, the growth assumptions on $u$ in \cite{farinaserrin, pucciserrin_2, dambrosiomitidieri_2} are quite different from \eqref{opequeno}. In this respect, we emphasize that Theorem \ref{teo_main_2} does not require any restriction on $Q$, while on the other hand, in \cite{farinaserrin, pucciserrin_2, dambrosiomitidieri_2}, no bounds on $\chi,\mu$ of the type in \eqref{pararange_2} are needed. 
}
\end{remark}

\begin{remark}\label{rem_farinaserrinliou}
\emph{Other Liouville type results for slowly growing solutions of \eqref{main_case_2}, skew with Theorem \ref{teo_main_2}, have been proved in \cite{farinaserrin, farinaserrin_I}. There, the authors assume $0 \le \chi < p-1$ and that \eqref{assu_barf} be replaced by the stronger $tf(t) \ge C|t|^{\omega+1}$, for some $\omega>0$.
Theorems 11 and 12 in \cite{farinaserrin} consider, respectively, the case $\omega<p-\chi-1$ and $\omega = p-\chi-1$, and for each one the study involves various sub-cases, notably those with $\mu <p-\chi$. It is worth to remark that in all but one sub-case the dimension $Q$ plays a role. The unique exception is when $\omega = p-\chi-1$ and $\mu < p-\chi$, for which it is shown that any $u$ solving \eqref{main_case_2} with the equality sign is constant provided that it grows polynomially.
}
\end{remark}

\begin{remark}[\textbf{Operators admitting multiple values of $p$}]\label{rem_multiple}
\emph{Suppose that $\varphi$ satisfies (WpC) in \eqref{assumptions} for two different values $1<p_m< p_M$. Clearly, $\varphi$ enjoys (WpC) for each $p \in [p_m,p_M]$. This is the case, for instance, of the mean curvature operator, which satisfies the assumptions of Theorems \ref{teo_main} and \ref{teo_main_2} for each $p \in (1,2]$. Although it might seem that conditions \eqref{pararange} and \eqref{assu_main} (and, analogously \eqref{pararange_2} and \eqref{assum_main_2}) are skew as we vary $p \in [p_m,p_M]$, we are going to show that the choice $p=p_M$ is the best one, that is, the one granting the largest range of parameters. To see this in the setting of Theorem \ref{teo_main} (that of Theorem \ref{teo_main_2} being analogous), define $\eta = p-1-\chi$. Then, \eqref{pararange} and \eqref{assu_main} can be rewritten as
$$
\begin{array}{c}
0 \le \eta \le p-1, \qquad \mu < \eta+1, \qquad \omega > \eta \\[0.3cm]
b \ge C(1+r)^{-\mu}, \qquad \disp f(t) \ge Ct^\omega, \qquad l(t) \ge C \varphi(t)t^{-\eta}.
\end{array}
$$
Since $p$ just appears as an upper bound for $\eta$, a larger $p$ guarantees a greater range for $\eta$ (hence, for $\mu$). Note also that a larger $\eta$ allow a smaller lower bound for $l$ at infinity.
}
\end{remark}

Both Theorems \ref{teo_main} and \ref{teo_main_2} are very much in the spirit of the recent \cite{dambrosiomitidieri_2}, whose method rests on a reduction procedure based on Kato's inequality. We recall that, if $u$ solves \eqref{lageneral}, Kato's inequality is the inequality satisfied by $u_\gamma \doteq (u-\gamma)_+$ ($\gamma$ being a constant, typically $\gamma=0$). Our approach is similar, but we do not need to find a Kato inequality, we directly investigate a superlevel set $\{u>\gamma\}$. Our choice here is to emphasize the next aspect: in the case $l(0)>0$, if $u^*$ is attained and assuming that everything be smooth enough, evaluating \eqref{main_case_2} at a maximum point gives $f(u^*) \le 0$. Therefore, property $f(u^*)\le 0$ can be thought as the validity of a \emph{weak maximum principle at infinity}, according to the point of view adopted in \cite{rigolisalvatorivignati_3, PRS_MEM}, see also \cite{marivaltorta}.\par

In Theorems \ref{teo_main} and \ref{teo_main_2}, the smaller is $\varphi(t)t^{1-p}$ at infinity, the weaker is our requirement on $l(t)$ as $t \ra +\infty$, and this is particularly suited for mean curvature type operators. For this reason, we state the following direct corollary of Theorem \ref{teo_main}. Note that the mean curvature operator satisfies the assumptions in Theorem \ref{teo_main} for each $p \in (1,2]$, but in view of Remark \ref{rem_multiple} it is enough to choose $p=2$.

\begin{corollary}\label{teo_meancurvature}
Let $\GG$ be a Carnot group with homogeneous norm $r$, and fix $\mu,\chi,\omega \in \R$ such that 
\begin{equation}\label{ipo_chilsigma}
0 \le \chi \le 1, \qquad \mu < 2-\chi, \qquad \omega> 1-\chi.
\end{equation}
Consider $b,f,l$ satisfying \eqref{assumptions} and
\begin{equation}\label{assunzbfl_meancurv}
\begin{array}{ll}
\disp b \ge C(1+r)^{-\mu} & \qquad \text{on } \, \GG, \\[0.3cm]
\disp f(t) \ge Ct^\omega & \qquad \text{for } \, t \ \text{ large enough}, \\[0.2cm]
\disp l(t) \ge C\frac{t^{\chi}}{1+t} & \qquad \text{on } \, \R^+,
\end{array}
\end{equation}
for some constant $C>0$. Then, for each solution $u \in \lip_\loc(\GG)$ of 
\begin{equation}\label{meancurv_case}
\diver_0 \left( \frac{\nabla_0 u}{\sqrt{1+|\nabla_0 u|^2}}\right) \ge b(x)f(u)l(|\nabla_0u|) \qquad \text{on } \, \GG
\end{equation}
it holds $u^* < +\infty$, and $f(u^*) \le 0$ unless $u$ is constant. If further 
\begin{equation}
tf(t) \ge C|t|^{\omega+1} \qquad \text{for } \, |t| \, \text{ large enough,} 
\end{equation}
then each solution $u$ of \eqref{meancurv_case} with the equality sign satisfies $u \in L^\infty(\GG)$ and, if $u$ is not constant, $f(u^*) \le 0 \le f(u_*)$. 
\end{corollary}

\begin{remark}\label{rem_belluz}
\emph{The third requirement in \eqref{ipo_chilsigma} is the sharp condition on $\omega$ to guarantee \eqref{KO} with $K(t)$ as in \eqref{modi_KO}. Indeed, because of \eqref{assunzbfl_meancurv}
\begin{equation}\label{calculis}
K(t) \lesssim t^{2-\chi}, \quad F(t) \gtrsim  t^{\omega+1}, \quad \frac{1}{K^{-1} \circ F} \lesssim t^{-\frac{\omega+1}{2-\chi}}  \qquad \text{as } \, t \ra +\infty.
\end{equation}
Note also that, if $\chi \in (0,1)$, $l(t)$ is allowed to go to zero both as $t \ra 0$ and as $t \ra +\infty$.
}
\end{remark}

Comparing with Theorem \ref{teo_farinaserrin} in the case of the mean curvature operator, and with $l(t)$ as in \eqref{assunzbfl_meancurv}, we observe that $\underline{\chi}$ in \eqref{parafarina} is $\chi-1$, whence the third in \eqref{condifarina} is met whenever $\underline{\omega} > 2-\chi$, a condition more demanding than $\omega >1-\chi$ in \eqref{ipo_chilsigma}. Furthermore, as in \cite{DAmbrMit, dambrosiomitidieri_2}, in our result the behaviour of $f(t)$ near $t=0$ is irrelevant.\par
Next, we focus on the link with Theorem \ref{teo_JFA}. First, in Corollary \ref{teo_meancurvature} none of $f,l$ is required to be $C$-increasing, and the range of $\mu$ in \eqref{ipo_chilsigma} is larger than the one in \eqref{range_JFA}. As stressed in Remark \ref{rem_strange!}, when $f(t) \ge Ct^\omega$ the conclusion of Theorem \ref{teo_JFA} holds if $\omega>-\chi$ (and $\omega \ge 0$), which is weaker than $\omega> 1-\chi$. This shift can be explained as follows: if we want to include the factor $(1+t)^{-1}$ in \eqref{assunzbfl_meancurv} in order to weaken the requirements on $l(t)$, the price to pay is an additional power in the asymptotic behaviour of $K(t)$ in \eqref{modi_KO}, and consequently the Keller-Osserman condition \eqref{KO} is more demanding on $\omega$. Note that we will produce counterexamples to show that $\omega> 1-\chi$ is sharp in the generality of Theorem \ref{teo_meancurvature}.\par
Eventually, we compare Theorem \ref{teo_meancurvature} with Corollary 1.4 in \cite{filippucci}. There, as said in Remark \ref{rem_fili}, solutions of \eqref{meancurv_case} with $l(t) = t^\chi$ are shown to be non-positive or constant provided that \eqref{range_filippucci} holds and $p \in (1,Q)$. In this respect, although the result grasps the sharp bound for $\mu$ in \eqref{ipo_chilsigma}, the above discussion now shows that condition $\omega> 1-\chi$ is \emph{not} sharp when $l(t) = t^\chi$, the sharp one being $\omega \ge 0$, $\omega > -\chi$.\\
\par

Theorem \ref{teo_main} covers the range
$$
0 \le \chi \le p-1, \qquad \mu < p-\chi, \qquad \omega> p-1-\chi.
$$
When $\chi > p-1$, namely, when $l(t)$ grows fast at infinity, things are quite different, 
as confirmed by the \emph{a-priori} estimate in \cite{dambrosiomitidieri_2} (Theorem 11.4), which fits very well with our results. There, the conclusions of Theorem \ref{teo_main} are shown to hold when \eqref{assu_main} is in force with $l(t) \ge t^\chi$ and $\omega = 0$, provided that
\begin{equation}\label{paracomchi}
\mu < 1, \qquad p-1 < \chi \le \left[\frac{Q-\mu}{Q-1}\right](p-1).
\end{equation}
Note that, as shown in Remark 11.8 of \cite{dambrosiomitidieri_2}, when $\mu=0$ the exponent $\frac{Q-\mu}{Q-1}(p-1)$ in \eqref{paracomchi} is sharp. Related interesting results, for possibly singular $b(x)$ and still in the range $\chi >p-1$, are given in \cite{lili}.

We conclude the paper with an appendix containing a version of the pasting lemma. This result enables to glue two solutions of \eqref{main_case_2} into a single one, even if one of them is not defined on the whole space. Since the only proof that we know is the one in \cite{Le} for $\R^n$, and contains an inessential assumption in our setting, we felt convenient to include a detailed argument. The pasting lemma is used below to produce counterexamples that show the sharpness of the parameters range in Theorems \ref{teo_main} and \ref{teo_main_2}.

\section{Proof of Theorem \ref{teo_main_2}}

The first step to prove Theorems \ref{teo_main} and \ref{teo_main_2} is to find a refined maximum principle for slowly growing solutions of 
$$
\Delta^\varphi_{\GG}u \ge K(1+r)^{-\mu}\frac{\varphi(|\gru|)}{|\gru|^{p-1-\chi}}
$$
on superlevel sets $\{u>\gamma\}$.

\begin{theorem}\label{teo_maximum}
Let $\GG$ be a Carnot group with homogeneous dimension $Q \ge 1$ and homogeneous norm $r$. Consider a $\varphi$-Laplace operator $\Delta_{\GG}^\varphi$ whose function $\varphi$ satisfies (WpC) in assumptions \eqref{assumptions} with $p>1$.\\
Fix $\mu, \sigma, \chi \in \R$ with the property that 
\begin{equation}\label{assu_musigmachi}
0 \le \chi < p-1, \qquad \mu \le p-\chi, \qquad 0 \le \sigma \le \sigma^* \doteq \frac{p-\chi-\mu}{p-\chi-1}, 
\end{equation}
and a function $u \in \lip_\loc(\GG)$ for which 
\begin{equation}\label{sup_u}
\hat u \doteq \limsup_{r(x) \ra +\infty} \frac{u_+(x)}{r(x)^\sigma} < +\infty.
\end{equation}
If, for some $\gamma \in \R$, the open set $\Omega_\gamma \doteq \{ u > \gamma\}$ is non-empty and $u$ is a non-constant, weak solution of 
\begin{equation}\label{ineq_superlevel_22}
\Delta^\varphi_{\GG}u \ge K(1+r)^{-\mu}\frac{\varphi(|\gru|)}{|\gru|^{p-1-\chi}} \qquad \text{on } \, \Omega_\gamma,
\end{equation}
then 
\begin{equation}
K \le H \cdot \hat u^{p-\chi-1},
\end{equation}
where $H = H(\sigma, \chi,p,\mu)$ is a constant satisfying
\begin{equation}\label{casi_C}
\begin{array}{l}
H=0 \quad \text{if either} \quad \left\{ \begin{array}{l}
\sigma < \sigma^*, \quad \text{or } \\[0.2cm] 
\sigma = \sigma^* = 0, \quad \text{or } \\[0.2cm] 
\sigma = \sigma^* > 0, \ \ (p-1)(\sigma-1) \le 1-Q 
\end{array}\right. \\[1cm]
\disp H = \sigma^{p-\chi-1}\big[(p-1)(\sigma-1)+Q-1\big] \quad \text{if } \, \sigma= \sigma^* >0, \ \  (p-1)(\sigma-1) > 1-Q.
\end{array}
\end{equation}

\end{theorem}

\begin{proof}
Without loss of generality, we can rescale $r$ in such a way that $|\nabla_0r| \le 1$. Set for convenience
\begin{equation}\label{def_Gteo}
S(t) \doteq \frac{t^{p-1}}{\varphi(t)},
\end{equation}
and note that, by (WpC), $S_* \doteq \inf_\GG S >0$. Inequality \eqref{ineq_superlevel_22} can be rewritten as
\begin{equation}\label{ineq_superlevel}
\Delta^\varphi_{\GG}u \ge K(1+r)^{-\mu}\frac{|\gru|^\chi}{S(|\gru|)}.
\end{equation}
Moreover, we can suppose $K>0$, otherwise the estimate is trivial. Note that \eqref{ineq_superlevel} is invariant with respect to translations $u \mapsto u_s \doteq u+s$. Fix $b > \hat u$. We claim that a suitable translated $u_s$ satisfies 
\begin{equation}
u_s \le b(1+r)^\sigma \qquad \text{on } \, M, \qquad u_s>0 \ \text{ somewhere.}
\end{equation}
Indeed, if $\sigma>0$, \eqref{sup_u} implies that $u < b(1+r)^\sigma$ outside a large compact set $\Omega$, and translating $u$ downwards we can achieve the same inequality also in $\Omega$, still keeping $u_s>0$ somewhere. On the other hand, the claim is obvious if $\sigma=0$. In this last case, note that here we do \emph{not} claim that $\hat u$ is not attained: this would follow from a strong maximum principle, which seems to be not known under the sole assumption (WpC). Using that the resulting $u_s$ is positive somewhere, we can also assume $\gamma\ge 0$. Hereafter, computations will be performed with $u=u_s$. Set for convenience 
\begin{equation}\label{def_eta}
\eta \doteq \mu + (\sigma-1)(p-\chi),
\end{equation}
and note that
\begin{equation}\label{ine_eta}
\begin{array}{lcl}
\disp \sigma \le \frac{p-\chi-\mu}{p-\chi-1} & \Longleftrightarrow & \sigma \ge \eta.
\end{array}
\end{equation}
Choose $\alpha>b$ and define
$$
v(x) = \alpha\big(1+r(x)\big)^\sigma - u(x),
$$
so that
\begin{equation}\label{def_v}
(\alpha-b)(1+r)^\sigma \le v \le \alpha(1+r)^\sigma \qquad \text{on } \, \Omega_\gamma. 
\end{equation}
Fix a function $\lambda \in C^1(\R)$ such that
$$
0 \le \lambda \le 1, \quad \lambda \equiv 0 \ \  \text{ on } \, (-\infty, \gamma], \qquad \lambda>0 \ \ \text{on } \, (\gamma, \infty), \quad \lambda' \ge 0,
$$
and a cut-off function $\zeta \in \lip_c(M)$ to be specified later, whose support has nontrivial intersection with $\Omega_\gamma$. Next, consider $F \in C^1(\R^2)$, $F=F(r,v)$, satisfying
\begin{equation}\label{propriety_F}
F(r,v)>0, \qquad F_v \doteq \frac{\partial F}{\partial v}(r,v) < 0.
\end{equation}
We insert the test function
$$
\zeta^{p} \lambda(u) F(v, r)  \in \lip_c(\overline{\Omega}_\gamma)
$$
in the weak definition of \eqref{ineq_superlevel}. Using $\lambda' \ge 0$ together with the Cauchy-Schwarz inequality, $|\grr| \le 1$ and taking into account \eqref{def_Gteo}, we obtain
\begin{equation}
\begin{array}{l}
\disp K \int \zeta^{p} \lambda F (1+r)^{-\mu}\frac{|\gru|^\chi}{S(|\gru|)} \le \disp p\int \zeta^{p-1} \lambda F \varphi(|\gru|)|\grho| + \int \zeta^p \lambda F_v\varphi(|\gru|)|\gru| \\[0.5cm]
\disp \qquad \qquad + \int \zeta^{p}\lambda \varphi(|\gru|)\left|\alpha\sigma (1+r)^{\sigma-1}F_v + F_r\right|  \\[0.5cm]
\qquad \qquad \le \disp p \int \zeta^{p-1} \lambda F \frac{|\gru|^{p-1}}{S(|\gru|)}|\grho| + \int \zeta^p \lambda F_v \frac{|\gru|^p}{S(|\gru|)} \\[0.5cm]
\disp \qquad \qquad + \int \zeta^{p}\lambda \frac{|\gru|^{p-1}}{S(|\gru|)} \left|\alpha\sigma (1+r)^{\sigma-1}F_v + F_r\right|. 
\end{array}
\end{equation}
Rearranging,
\begin{equation}\label{basepoint}
\begin{array}{l}
\disp \int \zeta^{p} \lambda \left|F_v\right| B(x,u) \le p\int \zeta^{p-1} \lambda F \frac{|\gru|^{p-1}}{S(|\gru|)}|\grho|,
\end{array}
\end{equation}
with 
\begin{equation}\label{def_B}
\begin{array}{lcl}
B(x,u) & = & \disp K(1+r)^{-\mu} \frac{F}{|F_v|}\frac{|\gru|^\chi}{S(|\gru|)} + \frac{|\gru|^p}{S(|\gru|)} \\[0.5cm]
& & \disp - \frac{|\gru|^{p-1}}{S(|\gru|)}\left|-\alpha\sigma (1+r)^{\sigma-1} + \frac{F_r}{|F_v|}\right| \\[0.5cm]
 & = & \disp \frac{|\gru|^\chi}{S(|\gru|)}\left[K(1+r)^{-\mu} \frac{F}{|F_v|} + |\gru|^{p-\chi}\right. \\[0.5cm]
& & \disp \left. - |\gru|^{p-1-\chi}\left|-\alpha\sigma (1+r)^{\sigma-1} + \frac{F_r}{|F_v|}\right| \right]. 
\end{array}
\end{equation}
Let assume the validity of the following
\begin{equation}\label{lower_boundB}
\emph{claim: } \quad B(x,u) \ge \Lambda \frac{|\gru|^p}{S(|\gru|)} \qquad \text{for some  $\Lambda>0$ independent of $r$.} 
\end{equation}
Plugging into \eqref{basepoint} gives
\begin{equation}
\begin{array}{l}
\disp \frac{\Lambda}{p} \int \zeta^{p} \lambda \left|F_v\right| \frac{|\gru|^p}{S(|\gru|)} \le \int \zeta^{p-1} \lambda F \frac{|\gru|^{p-1}}{S(|\gru|)}|\grho|,
\end{array}
\end{equation}
and thus, by H\"older inequality, 
\begin{equation}\label{basestep}
\begin{array}{lcl}
\disp \frac{\Lambda^p}{p^p} \int \zeta^{p} \lambda \left|F_v\right| \frac{|\gru|^p}{S(|\gru|)} & \le 
& \disp \int \frac{\lambda F}{S(|\gru|)} \left( \frac{F}{|F_v|}\right)^{p-1}|\grho|^p \\[0.5cm]
& \le & \disp S_*^{-1} \int \lambda F \left( \frac{F}{|F_v|}\right)^{p-1}|\grho|^p.
\end{array}
\end{equation}
Fix $R_0$ large enough that $u$ is not constant on $\Omega_\gamma \cap B_{R_0}\neq \emptyset$. We claim that the horizontal gradient $\nabla_0 u$ is not identically zero on $\Omega_\gamma \cap B_{R_0}$. Otherwise, since the first layer generates the whole Lie Algebra of $\GG$, the whole Euclidean gradient of $u$ would be zero, hence $u$ would be constant on connected components of $\Omega_\gamma$. By the very definition of $\Omega_\gamma$, this would imply that $\Omega_\gamma \equiv \GG$ and $u$ be constant, contradiction. For $R > 2R_0$, we choose $\zeta$ in such a way that
\begin{equation}\label{specipsi}
0 \le \zeta \le 1, \quad \zeta \equiv 1 \ \text{ on } B_{R}, \quad \supp(\zeta) \subset B_{2 R}, \quad |\grho| \le \frac{C}{R},
\end{equation}
for some absolute constant $C$. 
Inserting into \eqref{basestep} and recalling that $\lambda \le 1$ we obtain 
\begin{equation}\label{quasifinale}
\begin{array}{lcl}
\disp \frac{\Lambda^p}{p^p} \int_{\Omega_\gamma \cap B_{R_0}} \lambda \left|F_v\right| \frac{|\gru|^p}{S(|\gru|)} & \le & \disp \frac{C^p}{S_*R^p} \int_{(B_{2 R}\backslash B_R)\cap \Omega_\gamma} F \left( \frac{F}{|F_v|}\right)^{p-1}.
\end{array}
\end{equation}
We now need to check the validity of the claim in \eqref{lower_boundB}, for a suitable choice of $F$. Observe that the expression in square brackets in \eqref{def_B} is a function of the type 
$$
g(s)=\rho + s^{p-\chi} - \beta s^{p-\chi-1}, 
$$
for $s = |\gru|$ and positive parameters 
$$
\rho \doteq K(1+r)^{-\mu} \frac{F}{|F_v|}, \qquad \beta \doteq \left|-\alpha\sigma (1+r)^{\sigma-1} + \frac{F_r}{|F_v|}\right|
$$
depending on $r$. It is a calculus exercise to check that $g(s) \ge \Lambda s^{p-\chi}$ on $\mathbb R^+_0$ when
\begin{equation}\label{ine_lambda}
\Lambda \le 1 - \frac{p-\chi-1}{(p-\chi)^{\frac{p-\chi}{p-\chi-1}}}\left( \frac{\beta^{p-\chi}}{\rho}\right)^{1/(p-\chi-1)}.
\end{equation}
Inequality \eqref{ine_lambda} yields \eqref{lower_boundB} provided that $\beta^{p-\chi}/\rho$ is bounded from above by a (small enough) quantity independent of $r$. This suggests our choice of $F$, that will be different from case to case.\\[0.2cm] 
\noindent \emph{First case: } $\sigma>\eta$ (that is, $\sigma < \sigma^*$).\\
We choose 
\begin{equation}
F(v,r) = \exp\left\{ - \tau v(1+r)^{-\eta}\right\},
\end{equation}
for a real number $\tau>0$ to be specified later. Then, on $\Omega_\gamma$
$$
\frac{F}{|F_v|} = \frac{(1+r)^\eta}{\tau}, \qquad  \frac{F_r}{|F_v|} = \frac{v\eta}{(1+r)},
$$
and hence, by \eqref{def_v} and using $\sigma>\eta$,
$$
\begin{array}{ll}
\disp -\alpha(\sigma-\eta)(1+r)^{\sigma-1} \le -\alpha\sigma (1+r)^{\sigma-1} + \frac{F_r}{|F_v|} \le 0 & \quad \text{if } \, \eta<0. \\[0.5cm]
\disp \left[-\alpha(\sigma-\eta)-\eta b\right](1+r)^{\sigma-1} \le -\alpha\sigma (1+r)^{\sigma-1} + \frac{F_r}{|F_v|} \le 0 & \quad \text{if } \, \eta \ge 0
\end{array}
$$
Plugging into \eqref{def_B} we get
\begin{equation}\label{lowbounds_B}
B(x,u) \ge \disp \frac{|\gru|^\chi}{S(|\gru|)}\left[\frac{K}{\tau}(1+r)^{\eta-\mu} + |\gru|^{p-\chi} - |\gru|^{p-1-\chi}\beta^*(1+r)^{\sigma-1} \right], 
\end{equation}
with 
$$
0 < \beta^* = \left\{ \begin{array}{ll}
\alpha(\sigma-\eta) & \quad \text{if } \, \eta<0 \\[0.3cm]
\alpha(\sigma-\eta)+ \eta b & \quad \text{if } \, \eta \ge 0.
\end{array}\right.
$$
In view of identity \eqref{def_eta}, inequality \eqref{ine_lambda} applied with 
$$
\rho \doteq \frac{K}{\tau}(1+r)^{\eta-\mu}, \qquad \beta \doteq \beta^*(1+r)^{\sigma-1}
$$
gives
\begin{equation}\label{ine_lambda_2}
\Lambda \le 1 - (p-\chi-1)\left(\frac{\beta^*}{p-\chi}\right)^{\frac{p-\chi}{p-\chi-1}}\left( \frac{\tau}{K}\right)^{\frac{1}{p-\chi-1}}.
\end{equation}
For $\theta \in (0,1)$, choosing $\Lambda,\tau$ in such a way that 
$$
\Lambda = 1-\theta, \qquad \tau = \theta^{p-\chi-1} \frac{K}{(\beta^*)^{p-\chi}}\frac{(p-\chi)^{p-\chi}}{(p-\chi-1)^{p-\chi-1}},
$$
then \eqref{ine_lambda_2} is met with the equality sign and the claim in \eqref{lower_boundB} is proved. By our choice of $F$, \eqref{quasifinale} becomes
\begin{equation}\label{quasifinale_0}
\begin{array}{lcl}
\disp \left(\frac{\tau\Lambda}{p}\right)^p \int_{B_{R_0}\cap \Omega_\gamma} \lambda F(1+r)^{-\eta} \frac{|\gru|^p}{S(|\gru|)} & \le & \disp \frac{C^p}{S_*R^p} \int_{(B_{2R}\backslash B_R)\cap \Omega_\gamma} F(1+r)^{|\eta|(p-1)}.
\end{array}
\end{equation}
However, on $(B_{2R}\backslash B_R)\cap \Omega_\gamma$, since $\sigma-\eta >0$, \eqref{def_v} gives
$$
F(v,r) \le \exp\left(-\tau(\alpha-b)\big( 1+ R\big)^{\sigma-\eta}\right), \qquad (1+r)^{|\eta|(p-1)} \le (1+2R)^{|\eta|(p-1)}.
$$
Inserting into \eqref{quasifinale_0} and using \eqref{volgrowth}, we eventually get
\begin{equation}\label{quasifinale_2}
\begin{array}{lcl}
0 & < & \disp \left(\frac{\tau\Lambda}{p}\right)^{p} \int_{\Omega_\gamma \cap B_{R_0}}\lambda F(1+r)^{-\eta}\frac{|\gru|^p}{S(|\gru|)}\\[0.5cm]
& \le & \disp  \frac{C^{p}}{S_*R^{p}}\exp\left(-\tau(\alpha-b)\big( 1+ R\big)^{\sigma-\eta}\right)(1+ 2R)^{|\eta|(p-1)}\vol(B_{2R}) \\[0.5cm]
& & \longrightarrow 0 \quad \text{as } \, R \ra +\infty.
 \end{array}
\end{equation}
Therefore, letting $R\ra +\infty$ we deduce 
$$
\int_{\Omega_\gamma \cap B_{R_0}}\lambda F(1+r)^{-\eta}\frac{|\gru|^p}{S(|\gru|)} = 0.
$$
However, since $\gru \not\equiv 0$ on $\Omega_\gamma \cap B_{R_0}$ and $F>0$, the above integral cannot be zero. Concluding, when $\sigma<\sigma^*$ our assumption $K>0$ leads to a contradiction, hence $K\le 0$, as was to be proved.\\[0.2cm]
\noindent \emph{Second case: } $\sigma=\eta$ (that is, $\sigma = \sigma^*$).\\
In this case, by \eqref{def_eta} it holds $\sigma-\mu=(\sigma-1)(p-\chi)$. We choose 
\begin{equation}
F(v,r) = v^{-\tau}, 
\end{equation}
$\tau>0$ to be determined. Then, using \eqref{def_v},
\begin{equation}\label{lowbounds_Bpol}
\begin{array}{lcl}
B(x,u) & \ge & \disp \frac{|\gru|^\chi}{S(|\gru|)}\left[\frac{K}{\tau}(1+r)^{-\mu}v + |\gru|^{p-\chi}\right. \\[0.5cm]
& & \disp \left. - |\gru|^{p-1-\chi}\alpha\sigma(1+r)^{\sigma-1} \right] \\[0.5cm]
 & \ge & \disp \frac{|\gru|^\chi}{S(|\gru|)}\left[\frac{K(\alpha-b)}{\tau}(1+r)^{\sigma-\mu} + |\gru|^{p-\chi}\right. \\[0.5cm]
& & \disp \left. - |\gru|^{p-1-\chi}\alpha\sigma(1+r)^{\sigma-1} \right], 
\end{array}
\end{equation}
and, by \eqref{ine_lambda}, \eqref{lower_boundB} holds whenever
\begin{equation}\label{ine_lambda_4}
\Lambda \le 1 - (p-\chi-1)\left(\frac{\alpha\sigma}{p-\chi}\right)^{\frac{p-\chi}{p-\chi-1}}\left( \frac{\tau}{K(\alpha-b)}\right)^{\frac{1}{p-\chi-1}}.
\end{equation}
For $\theta \in (0,1)$, set 
\begin{equation}\label{policaso}
\Lambda = 1-\theta, \qquad \tau = \theta^{p-\chi-1} \frac{(p-\chi)^{p-\chi}}{(p-\chi-1)^{p-\chi-1}} \frac{K(\alpha-b)}{(\alpha \sigma)^{p-\chi}}
\end{equation}
in order to satisfy \eqref{ine_lambda_4} with the equality sign. Inequality \eqref{quasifinale} now reads 
\begin{equation}\label{quasifinale_5}
\begin{array}{lcl}
\disp \frac{\Lambda^p}{p^p} \int_{\Omega_\gamma \cap B_{R_0}} \lambda \left|F_v\right| \frac{|\gru|^p}{S(|\gru|)} & \le & \disp \frac{C^p}{S_*R^p} \int_{(B_{2R}\backslash B_R)\cap \Omega_\gamma} v^{-\tau} \left( \frac{v}{\tau}\right)^{p-1} \\[0.5cm]
 & \le & \disp \frac{C^p}{\tau^{p-1}R^p} \int_{(B_{2R}\backslash B_R)\cap \Omega_\gamma} v^{p-1-\tau}. \\[0.5cm]
\end{array}
\end{equation}
Estimating $v$ with the aid of \eqref{def_v} and according to the sign of $p-1-\tau$, and using \eqref{volgrowth}, the right-hand side is bounded from above by 
$$
C_1R^{Q-p+ \sigma(p-1-\tau)}, 
$$
for a suitable constant $C_1>0$. Letting $R \ra +\infty$, for \eqref{quasifinale_5} to be compatible it is necessary that $Q - p + \sigma(p-1-\tau) \ge 0$, that is, 
\begin{equation}\label{finla}
\tau \sigma \le Q-1 + (p-1)(\sigma-1). 
\end{equation}
\begin{itemize}
\item[(i)] If $Q-1 + (p-1)(\sigma-1) <0$ or $Q-1 + (p-1)(\sigma-1)  = 0$ and $\sigma >0$, then there exists no $\tau>0$ satisfying the inequality. Therefore, $K>0$ leads to a contradiction, as required. This is the third case for $H=0$ in \eqref{casi_C}.

\item[(ii)] If $Q-1 + (p-1)(\sigma-1) \ge 0$ then, setting $\alpha = tb$ for $t > 1$ in expression \eqref{policaso} for $\tau$, inserting \eqref{policaso} into \eqref{finla} and solving \eqref{finla} with respect to $K$, letting $\theta \uparrow 1$ and $b \downarrow \hat u$ we deduce
$$
K \le  \left[Q-1 + (\sigma-1)(p-1)\right] \frac{(p-\chi-1)^{p-\chi-1}}{(p-\chi)^{p-\chi}}\hat u^{p-\chi-1}\sigma^{p-\chi-1} \frac{t^{p-\chi}}{t-1},
$$
and minimizing over $t \in (1,+\infty)$ we get
\begin{equation}\label{allafine}
K \le \left[Q-1 + (\sigma-1)(p-1)\right] \hat u^{p-\chi-1}\sigma^{p-\chi-1}.
\end{equation}
This concludes the case $Q-1 + (p-1)(\sigma-1) > 0$ and $\sigma= \sigma^*>0$. To deal with the remaining case $Q-1 + (p-1)(\sigma-1) \ge 0$ and $\sigma = \sigma^* = 0$, we can consider a downward translation $u_s$ of $u$ in place of $u$, and $\gamma = 0$. Then, $u_s$ satisfies \eqref{ineq_superlevel} with the same constant $K$, hence \eqref{allafine} holds for each $\hat u_s$. However, $\hat u_s$ can be made as small as we wish, and since we have assumed $K>0$ this would contradict \eqref{allafine}. Concluding, necessarily $K \le 0$, as required.
\end{itemize}
\end{proof}

\begin{remark} 
\emph{Observe that, while the techniques in \cite{MitidPohoz, DAmbrMit, farinaserrin, dambrosiomitidieri_2} seem to need a polynomial volume growth of balls to conclude sharp Liouville properties, our approach in Proposition \ref{teo_maximum} above is flexible enough to handle very general Riemannian manifolds, in particular, those for which $\vol(B_R)$ grows exponentially like the hyperbolic space. The Riemannian setting will be the subject of future investigation, see \cite{BMPR}.
}
\end{remark}

\subsection*{Sharpness of Proposition \ref{teo_maximum} and Theorem \ref{teo_main_2}}

We consider the mean curvature operator in $(\R^Q,\metric)$ with $Q \ge 2$, for which $p=2$ and $S(t) = \sqrt{1+t^2}$ in \eqref{def_Gteo}. Our aim is to produce solutions of
$$
\diver\left( \frac{\nabla u}{\sqrt{1+|\nabla u|^2}}\right) \ge \bar C(1+r)^{-\mu} \frac{|\nabla u|^\chi}{\sqrt{1+|\nabla u|^2}} \qquad \text{on } \, \R^Q
$$
satisfying $u(x) = \mathcal{O}\big(r(x)^\sigma\big)$, outside of the range where $H=0$ in \eqref{casi_C}. More precisely, we produce such solutions when:
\begin{equation}\label{complerange}
\disp 0 \le \chi <1, \qquad \mu < 2-\chi, \qquad \sigma \ge \sigma^* \doteq \frac{2-\chi-\mu}{1-\chi}.
\end{equation}
We underline that, in our setting, $(p-1)(\sigma-1) = \sigma -1 > 1-Q$. Hence, in the borderline case $\sigma = \sigma^*$ we are in the range complementary to the one that imply $H=0$.\par
Fix a smooth, non-decreasing  function $h \in C^\infty(\R^+_0)$, and consider $u(x) = h(r^2(x))$ with $r(x) = |x|$. Then, since $\Hess(r^2) = 2 \metric$,
$$
\nabla u = 2rh' \nabla r, \quad |\nabla u|^2= 4r^2(h')^2, \qquad \nabla \di u = 4r^2h'' \di r \otimes \di r + 2h' \metric,
$$
hence, using that $h' \ge 0$,
\begin{equation}\label{cumputa}
\begin{array}{l}
\disp \diver\left( \frac{\nabla u}{\sqrt{1+|\nabla u|^2}}\right) = \disp \frac{\Delta u}{\sqrt{1+|\nabla u|^2}} - \frac{\mathrm{Hess}u(\nabla u, \nabla u)}{(1+|\nabla u|^2)^{3/2}} \\[0.5cm]
\disp \quad = \frac{1}{\sqrt{1+|\nabla u|^2}} \left[ 2(Q-1)h' + \frac{4r^2h''+ 2h'}{1+4r^2(h')^2}\right]
\end{array}
\end{equation}
Choosing $h(t) = (1+t)^{\sigma/2}$, for some $\sigma>0$, we obtain  
\begin{equation}\label{basamento}
\disp \diver\left( \frac{\nabla u}{\sqrt{1+|\nabla u|^2}}\right) \ge \frac{\sigma(1+r^2)^{\frac{\sigma-4}{2}}}{\sqrt{1+|\nabla u|^2}} \left[ (Q-1)(1+r^2) + \frac{1+r^2(\sigma-1)}{1+\sigma^2r^2(1+r^2)^{\sigma-2}}\right]. 
\end{equation}
If $\sigma \ge 1$, we can get rid of the second term in square brackets. On the other hand, if $\sigma \in (0,1)$, 
$$
\begin{array}{lcl}
\disp \diver\left( \frac{\nabla u}{\sqrt{1+|\nabla u|^2}}\right) & \ge & \disp \frac{\sigma(1+r^2)^{\frac{\sigma-4}{2}}}{\sqrt{1+|\nabla u|^2}} \left[ (Q-1)(1+r^2) + (1+r^2)(\sigma-1)\right] \\[0.5cm]
& = & \disp  \frac{\sigma(1+r^2)^{\frac{\sigma-4}{2}}}{\sqrt{1+|\nabla u|^2}} (Q+ \sigma -2)(1+r^2). 
\end{array}
$$
Summarizing, for each $\sigma>0$ there exists a constant $C(\sigma, Q)$ such that 
\begin{equation}\label{semplice!}
\diver\left( \frac{\nabla u}{\sqrt{1+|\nabla u|^2}}\right) \ge \frac{C(1+r)^{\sigma-2}}{\sqrt{1+|\nabla u|^2}} \qquad \text{on } \, \R^Q.
\end{equation}
Now, consider parameters $\chi,\mu,\sigma$ in the range prescribed in \eqref{complerange}. Since 
$$
|\nabla u| = \sigma r(1+r^2)^{\frac{\sigma-2}{2}} \le \sigma(1+r^2)^{\frac{\sigma-1}{2}} \le C_1(1+r)^{\sigma-1}
$$
we deduce, by the third in \eqref{complerange},
\begin{equation}\label{gradu}
|\nabla u|^\chi(1+r)^{-\mu} \le C_2(1+r)^{\chi(\sigma-1)-\mu} \le C_3(1+r)^{\sigma-2},
\end{equation}
and inserting into \eqref{semplice!} we infer
\begin{equation}\label{semplice!!}
\diver\left( \frac{\nabla u}{\sqrt{1+|\nabla u|^2}}\right) \ge C_4(1+r)^{-\mu} \frac{|\nabla u|^\chi}{\sqrt{1+|\nabla u|^2}} \qquad \text{on } \, \R^Q,
\end{equation}
for some constant $C_4>0$. Therefore, if $Q \ge 2$ the range where $H=0$ in \eqref{casi_C} cannot be improved. 

\subsection*{Proof of Theorem \ref{teo_main_2}}

Suppose, by contradiction, that either $u$ is not bounded above or that $f(u^*)>0$. In both of the cases, we can find $\gamma < u^*$ such that $f(t) \ge C$ for $t>\gamma$, for some constant $C>0$ (in the first case, by using the second in \eqref{assum_main_2}). Therefore, because of our assumptions on $b,f,l$, $u$ turns out to be a non-constant solution of
$$
\Delta_\GG^\varphi u \ge K(1+r)^{-\mu}\frac{\varphi(|\gru|)}{|\gru|^{p-1-\chi}} \qquad \text{on } \, \{ u>\gamma\},
$$
for some $K>0$. Applying Proposition \ref{teo_maximum} and taking into account that, in our assumptions, $\sigma = \sigma^*>0$ and $\hat u=0$, we deduce that $K \le 0$, contradiction. As for the second part of Theorem \ref{teo_main_2}, if $u$ solves 
$$
\Delta_\GG^\varphi u = b(x) f(u)l(|\gru|),
$$
then $-u$ is a solution of 
$$
\Delta_\GG^\varphi v = b(x) \bar f(v)l(|\nabla_0v|), \qquad \text{with } \, \bar f(t) \doteq -f(-t).
$$
Because of \eqref{assu_barf}, $\bar f(t) \ge C$ for $t$ large, and we can apply again Proposition \ref{teo_maximum}, now to $-u$, to deduce that $(-u)$ is bounded above and $\bar f((-u)^*) \le 0$. In other words, $f(u_*) \ge 0$, which concludes the proof.\\
\par
We conclude by commenting on Remark \ref{rem_othercases}:
\begin{itemize}
\item[-] if $\mu = p-\chi$, then we can apply Proposition \ref{teo_maximum} with $\sigma = \sigma^*=0$ to deduce that $K \le 0$ (hence, all of our conclusions) provided that  \eqref{sup_u} holds, that is, if $u$ is bounded above;
\item[-] if \eqref{opequeno} is not satisfied, but still
$$
u_+(x) = \mathcal{O}\left( r(x)^{ \frac{p-\chi-\mu}{p-\chi-1}}\right) \qquad \text{as } \, r(x) \ra +\infty,
$$
then we are in the case $\sigma = \sigma^*>0$ and $\hat u>0$ of Theorem \ref{teo_main_2}. We obtain that $K \le 0$ whenever
$$
(\sigma^*-1)(p-1) \le 1-Q, \qquad \text{that is,} \qquad \frac{p-\chi-\mu}{p-\chi-1} \le 1 - \frac{Q-1}{p-1} = \frac{p-Q}{p-1},
$$
as claimed.
\end{itemize}

\section{Proof of Theorem \ref{teo_main}}

We begin with the following proposition. The idea of the proof is an adaptation of the one in Lemma 2.2 of \cite{farinaserrin}.

\begin{proposition}\label{lem_importante}
Let $\GG$ be a Carnot group with homogeneous dimension $Q \ge 1$ and homogeneous norm $r$. Consider a $\varphi$-Laplace operator $\Delta_{\GG}^\varphi$ whose function $\varphi$ satisfies (WpC) in assumptions \eqref{assumptions} with $p>1$.\\
Fix $\mu, \omega, \chi \in \R$ with the property that 
\begin{equation}\label{assu_musigmachi_33}
0 \le \chi \le p-1, \qquad \mu < p-\chi, \qquad \omega > p-\chi-1. 
\end{equation}
Then, for each $\gamma \ge 0$ and $K>0$, there exists no non-constant weak solution $u \in \lip_\loc(\GG)$ of 
\begin{equation}\label{ineq_superlevel_33}
\Delta^\varphi_{\GG}u \ge K(1+r)^{-\mu}u^\omega\frac{\varphi(|\gru|)}{|\gru|^{p-1-\chi}} \qquad \text{on } \, \Omega_\gamma \doteq \{u>\gamma\} \neq \emptyset.
\end{equation}
\end{proposition}

\begin{proof}
As in Proposition \ref{teo_maximum}, define for convenience
$$
S(t) \doteq \frac{t^{p-1}}{\varphi(t)},
$$
and note that $S_* \doteq \inf_\GG S >0$ in view of (WpC) in \eqref{assumptions}. Suppose that $u$ is a non-constant solution of \eqref{ineq_superlevel_33} for some $\gamma \ge 0$ and $K>0$, and take $\lambda \in C^1(\R)$ such that 
$$
0 \le \lambda \le 1, \quad \lambda'\ge 0, \quad \lambda\equiv 0 \, \text{ on } \, (-\infty,\gamma], \quad  \lambda>0 \, \text{ on } \, (\gamma, \infty). 
$$
Let $\psi \in C^\infty_c(\GG)$ be a cut-off function, and let $\eta,\alpha>1$ to be specified later. We plug the non-negative test function
$$
\phi = \psi^\eta \lambda(u)u^\alpha \in \lip_c(\GG)
$$
in the weak definition of \eqref{ineq_superlevel_33}, and we use $\lambda' \ge 0$, to obtain
\begin{equation}
\begin{aligned}
\disp K\int \psi^\eta \lambda \frac{u^{\alpha+\omega}}{(1+r)^\mu} \frac{|\gru|^\chi}{S(|\gru|)} & \le  - \disp \int \frac{\varphi(|\gru|)}{|\gru|} \langle \gru, \nabla_0(\psi^\eta \lambda u^\alpha) \rangle \\
& \le \disp \eta \int \psi^{\eta-1} \lambda u^\alpha \varphi(|\gru|)|\nabla_0 \psi| \\
 &\qquad -\alpha \int \psi^\eta \lambda u^{\alpha-1} \varphi(|\gru|) |\gru|.
\end{aligned}
\end{equation}
Hence, by \eqref{def_Gteo}, 
\begin{equation}\label{beginning}
\begin{aligned}
\disp K \int \psi^\eta \lambda \frac{u^{\alpha+\omega}}{(1+r)^\mu}\frac{|\gru|^\chi}{S(|\gru|)} & \le  \disp \eta \int \psi^{\eta-1} \lambda u^\alpha \frac{|\gru|^{p-1}}{S(|\gru|)}|\nabla_0\psi|\\
&\qquad- \alpha \int \psi^\eta \lambda u^{\alpha-1} \frac{|\gru|^{p}}{S(|\gru|)}.
\end{aligned}
\end{equation}
By \eqref{assu_musigmachi_33}, for each $\alpha>1$ it holds
\begin{equation}\label{alphadecente}
(p-1)(\alpha + \omega) > \chi\alpha.
\end{equation}
As in Lemma  2.2 in \cite{farinaserrin}, we now use the triple Young inequality to the first term in the right-hand side of \eqref{beginning}: we need to find $z_1,z_2,z_3>1$ satisfying 
\begin{equation}\label{expoholder}
\frac{1}{z_1} + \frac{1}{z_2} + \frac{1}{z_3} = 1
\end{equation}
and $\tau,\bar C>0$ such that
\begin{equation}\label{desejavel}
\psi^{\eta-1} \lambda u^\alpha \frac{|\gru|^{p-1}}{S(|\gru|)}|\nabla_0\psi| = \mathcal{J}_1^{\frac{1}{z_1}}\mathcal{J}_2^{\frac{1}{z_2}}\mathcal{J}_3^{\frac{1}{z_3}},
\end{equation}
with 
\begin{equation}\label{choices}
\begin{array}{lcl}
\mathcal{J}_1 & = & \disp \frac{K}{2\eta} \psi^\eta \lambda \frac{u^{\alpha + \omega}}{(1+r)^\mu} \frac{|\gru|^\chi}{S(|\gru|)} \\[0.5cm]
\mathcal{J}_2 & = & \disp \frac{\alpha}{\eta} \psi^\eta \lambda u^{\alpha-1} \frac{|\gru|^{p}}{S(|\gru|)} \\[0.5cm]
\mathcal{J}_3 & = & \disp \frac{\bar C}{\eta} (1+r)^{\tau} \frac{|\nabla_0 \psi|^{z_3}}{S(|\gru|)}.
\end{array}
\end{equation}
considering powers of $u$, $|\gru|$, $r$ and $\psi$, to obtain \eqref{desejavel} we need the following balancing:
$$
\begin{array}{rllrll}
i) & \text{powers of $u$:} & \ \disp \alpha = \frac{\alpha + \omega}{z_1} + \frac{\alpha-1}{z_2} & \quad ii) & \text{powers of $|\gru|$:} & \ \disp p-1 = \frac{\chi}{z_1} + \frac{p}{z_2} \\[0.5cm]
iii) & \text{powers of $r$:} &  \ \disp 0 = - \frac{\mu}{z_1} + \frac{\tau}{z_3} & \quad iv) & \text{powers of $\psi$:} & \ \disp \eta -1 = \frac{\eta}{z_1} + \frac{\eta}{z_2}. 
\end{array}
$$
Solving the first two equations with respect to $z_1$ and $z_2$, and then recovering $z_3$ from \eqref{expoholder}, we get
$$
\begin{array}{c}
\disp \frac{1}{z_1} = \frac{\alpha + p-1}{p(\alpha + \omega)-\chi(\alpha-1)}, \qquad \disp \frac{1}{z_2} = \frac{(p-1)(\alpha+\omega) - \chi \alpha}{p(\alpha+\omega) -\chi(\alpha-1)}, \\[0.5cm]
\disp \qquad \frac{1}{z_3} = \frac{\omega-(p-\chi-1)}{p(\alpha+\omega)-\chi(\alpha-1)}
\end{array}
$$
(these are positive numbers because of \eqref{alphadecente} and the Keller-Osserman condition $\omega > p-\chi-1$), and from the last two equations,
$$
\tau = \mu \frac{z_3}{z_1} = \mu \frac{\alpha+p-1}{\omega-(p-\chi-1)}, \qquad \eta = z_3.  
$$
The constant $\bar C$ is then uniquely determined by \eqref{desejavel}. Having found the right parameters, and since the triple Young inequality reads as
$$
\mathcal{J}_1^{\frac{1}{z_1}}\mathcal{J}_2^{\frac{1}{z_2}}\mathcal{J}_3^{\frac{1}{z_3}} \le \mathcal{J}_1 + \mathcal{J}_2 + \mathcal{J}_3,
$$
from \eqref{desejavel} and \eqref{choices} we deduce
$$
\begin{array}{lcl}
\disp \eta \psi^{\eta-1} \lambda u^\alpha \frac{|\gru|^{p-1}}{S(|\gru|)}|\nabla_0\psi| & \le & \disp \frac{K}{2} \psi^\eta \lambda \frac{u^{\alpha + \omega}}{(1+r)^\mu} \frac{|\gru|^\chi}{S(|\gru|)} + \alpha \psi^\eta \lambda u^{\alpha-1} \frac{|\gru|^{p}}{S(|\gru|)} \\[0.5cm]
& & + \disp \bar C (1+r)^{\mu \frac{z_3}{z_1}} \frac{|\nabla_0 \psi|^{z_3}}{S(|\gru|)}.
\end{array}
$$
Inserting into \eqref{beginning} and using that $S_* >0$ we get
\begin{equation}\label{beginning_3}
\begin{aligned}
\disp \frac{K}{2} \int \psi^\eta \lambda \frac{u^{\alpha+\omega}}{(1+r)^\mu}\frac{|\gru|^\chi}{S(|\gru|)} & \le \frac{\bar C}{S_*} \int  (1+r)^{\mu \frac{z_3}{z_1}} |\nabla_0 \psi|^{z_3}. 
\end{aligned}
\end{equation}
For large $R>0$, we choose as $\psi \in C^\infty_c(\GG)$ a cut-off function satisfying
$$
0 \le \psi \le 1, \quad \psi \equiv 1 \ \ \text{ on } \, B_R, \quad \psi \equiv 0 \ \ \text{ on } \, \GG\backslash B_{2R}, \quad |\nabla_0\psi| \le \frac{C}{R},
$$
for an absolute constant $C$. Since $\lambda=0$ when $u \le \gamma$, from \eqref{beginning_3} and \eqref{volgrowth} we obtain
\begin{equation}\label{farise}
\begin{array}{lcl}
\disp \frac{K}{2} \int_{B_R \cap \Omega_\gamma} \lambda \frac{u^{\alpha+\omega}}{(1+r)^\mu}\frac{|\gru|^\chi}{S(|\gru|)} & \le & \disp \frac{K}{2} \int \psi^\eta \lambda \frac{u^{\alpha+\omega}}{(1+r)^\mu}\frac{|\gru|^\chi}{S(|\gru|)} \\[0.5cm]
& \le & \disp \frac{\bar C}{S_*} \int_{B_{2R}} (1+r)^{\mu \frac{z_3}{z_1}} \left(\frac{C}{R}\right)^{z_3} \\[0.5cm]
& \le & \hat C R^{\mu \frac{z_3}{z_1} - z_3 + Q},
\end{array}
\end{equation}
The exponent of $R$ in \eqref{farise} can be written as 
$$
\mu \frac{z_3}{z_1} - z_3 + Q = \frac{\mu(\alpha+p-1) - p(\alpha+\omega) + \chi(\alpha-1) + Q(\omega-p+\chi+1)}{\omega-(p-\chi-1)},
$$
which is negative provided that
\begin{equation}\label{mau}
\alpha (\mu-p+\chi) < -Q(\omega-p+\chi+1) -\mu(p-1) +p\omega +\chi.
\end{equation}
Since, in our assumptions, $\mu < p-\chi$, for $\alpha$ large enough we can guarantee inequality \eqref{mau}. Fixing one such $\alpha$ and letting $R \ra +\infty$ in \eqref{farise}, from $K>0$ we deduce
\begin{equation}\label{attheend!}
\int_{\Omega_\gamma} \lambda \frac{u^{\alpha+\omega}}{(1+r)^\mu}\frac{|\gru|^\chi}{S(|\gru|)} \equiv 0. 
\end{equation}
However, since $\Omega_\gamma \neq \emptyset$, $\lambda>0$ on $(\gamma,+\infty)$, and $u$ is positive and non-constant on $\Omega_\gamma$, the integral in \eqref{attheend!} is strictly positive, a contradiction. 
\end{proof}

\subsection*{Proof of Theorem \ref{teo_main}}

The case when $u$ is constant has already been discussed in Remark \ref{rem_constants}, hence we just need to consider non-constant solutions. Fix $\gamma>0$ such that $f(t) \ge Ct^\omega$ on $[\gamma, +\infty)$. If $u^*= +\infty$, then by \eqref{assu_main} $u$ would be a non-constant solution of 
$$
\Delta_\GG^\varphi u \ge K(1+r)^{-\mu}u^\omega\frac{\varphi(|\gru|)}{|\gru|^{p-1-\chi}}  \qquad \text{on } \, \Omega_\gamma,
$$
for some $K>0$, which contradicts Proposition \ref{lem_importante}. Therefore, $u$ is bounded above. To prove that $f(u^*) \le 0$, suppose by contradiction that $f(u^*)>0$, and fix $\gamma<u^*$ close enough to $u^*$ in such a way that 
$$
f(u) \ge \frac{f(u^*)}{2} \doteq C_0 \qquad \text{on } \, \Omega_\gamma.
$$
Then, again by \eqref{assu_main} $u$ is a non-constant solution of 
\begin{equation}\label{sertao}
\Delta_\GG^\varphi u \ge K C_0(1+r)^{-\mu}\frac{\varphi(|\gru|)}{|\gru|^{p-1-\chi}}  \qquad \text{on } \, \Omega_\gamma,
\end{equation}
for some $K>0$. Since the above equation is invariant by translation, set $u_s \doteq u+s$, $s \in \R$. Up to choosing $s$ and $\gamma$ appropriately, we can ensure that   
$$
0< (u_s^*)^\omega < C_0, \qquad \gamma = 0.
$$
Inserting into \eqref{sertao}, we obtain that $u_s$ is a non-constant solution of 
$$
\Delta_\GG^\varphi u_s \ge K u_s^\omega(1+r)^{-\mu}\frac{\varphi(|\nabla_0u_s|)}{|\nabla_0u_s|^{p-1-\chi}}  \qquad \text{on } \, \{u_s>0\},
$$
again contradicting Proposition \ref{lem_importante}. Hence, $f(u^*) \le 0$. If $u$ solves \eqref{main_case} with the equality sign, and $f$ satisfies \eqref{equazione_fmain_simme}, then we can apply the first part of the proof both to $u$ and to $-u$, and proceeding as in the proof of Theorem \ref{teo_main_2} we obtain $u_* > -\infty$ and $f(u_*) \ge 0$, as required.

\subsection*{Sharpness of Proposition \ref{lem_importante} and Theorem \ref{teo_main}}

As we did for Proposition \ref{teo_maximum}, we prove the sharpness in the special case of the mean curvature operator in $\R^Q$, $Q \ge 2$. Thus, we want to exhibit unbounded solutions of 
\begin{equation}\label{aim}
\diver\left( \frac{\nabla u}{\sqrt{1+|\nabla u|^2}}\right) \ge C(1+r)^{-\mu} u^\omega \frac{|\nabla u|^\chi}{\sqrt{1+|\nabla u|^2}}
\end{equation}
on some superlevel set $\Omega_\gamma \doteq \{u>\gamma\}$, $\gamma>0$, outside of the range \eqref{pararange}, which for $p=2$ becomes 
	\begin{equation}\label{pararange_meancurv}
	0 \le \chi \le 1, \qquad  \mu<2-\chi, \qquad \omega>1-\chi.
	\end{equation}
Once we have an unbounded solution of \eqref{aim} on some $\Omega_\gamma$ with $\gamma>0$, we can easily produce a solution $v$ on the whole $\R^Q$ that does not satisfy the $L^ \infty$-estimate of Theorem \ref{teo_main}: first, choose $f \in C^0(\R)$ satisfying 
$$
f\equiv 0 \ \text{ on } \, (-\infty, 2\gamma), \quad f(t) \le t^\omega \ \text{ for } \, t \in [2\gamma, 3\gamma], \quad f(t) = t^\omega \quad \text{ for } t \ge 3\gamma, 	
$$	
and note that $u$ solves
\begin{equation}\label{eeerrr}
\diver\left( \frac{\nabla u}{\sqrt{1+|\nabla u|^2}}\right) \ge C(1+r)^{-\mu} f(u) \frac{|\nabla u|^\chi}{\sqrt{1+|\nabla u|^2}} \qquad \text{on } \, \{ u>2\gamma\}.
\end{equation}
Now, since $f(2\gamma)=0$, the constant function $2\gamma$ solves \eqref{eeerrr} with the equality sign on the whole $\R^Q$. By the pasting lemma (Proposition \ref{lem_pasting} below), the function
$$
v \doteq \left\{ \begin{array}{ll}
u & \quad \text{on } \, \{u>2\gamma\} \\[0.2cm]
2\gamma & \quad \text{otherwise,}
\end{array}\right.
$$
solves 
$$
\diver\left( \frac{\nabla v}{\sqrt{1+|\nabla v|^2}}\right) \ge C(1+r)^{-\mu} f(v) \frac{|\nabla v|^\chi}{\sqrt{1+|\nabla v|^2}} \qquad \text{on } \, \R^ Q.
$$
Being $v$ unbounded, it gives our desired counterexample.\\[0.2cm]
\noindent \emph{Counterexample in the range}
$$
0 \le \chi \le 1, \qquad \mu < 2-\chi, \qquad \omega < 1-\chi.
$$	
We consider the same function as the one used to show the sharpness of Proposition \ref{teo_maximum}, that is, $u(x) = h(r^2(x))$ with $h(t) = (1+t)^{\sigma/2}$ and some $\sigma>0$ to be determined. Since $Q \ge 2$, by \eqref{semplice!}
\begin{equation}\label{unun}
\diver\left( \frac{\nabla u}{\sqrt{1+|\nabla u|^2}}\right) \ge \frac{C(1+r)^{\sigma-2}}{\sqrt{1+|\nabla u|^2}} \qquad \text{on } \, \R^Q,
\end{equation}
for some constant $C>0$. Moreover, by the definition of $u$, 
\begin{equation}\label{theend}
u^\omega |\nabla u|^\chi(1+r)^{-\mu} \le C_1(1+r)^{\sigma \omega + \chi(\sigma-1)-\mu}.
\end{equation}
Define for convenience $\delta>0$ according to the identity $\omega = 1-\chi-\delta$. The inequality 
\begin{equation}\label{ineccetera}
\sigma\omega +\chi(\sigma-1)-\mu \le \sigma-2 
\end{equation}
is satisfied if and only if
\begin{equation}\label{sigmagrande}
\sigma \delta \ge 2-\chi-\mu,
\end{equation}
which holds provided $\sigma$ is large enough. Hence, if $\sigma$ is sufficiently large, from \eqref{unun}, \eqref{theend} and \eqref{ineccetera} we obtain that $u$ is an unbounded solution of
\begin{equation}\label{cisiamo!!!}
\diver\left( \frac{\nabla u}{\sqrt{1+|\nabla u|^2}}\right) \ge C_8(1+r)^{-\mu} u^\omega \frac{|\nabla u|^\chi}{\sqrt{1+|\nabla u|^2}} \qquad \text{on } \, \R^Q.
\end{equation}
Note that we admit the possibility that $\omega<0$.\\[0.2cm]
\noindent \emph{Counterexample in the range}
$$
0 \le \chi \le 1, \qquad \mu \ge 2-\chi, \qquad \omega = 1-\chi.
$$	
We still use the above counterexample: it is enough to observe that \eqref{ineccetera}, equivalently \eqref{sigmagrande}, is met for $\omega=1-\chi$ (i.e. $\delta =0$) if and only if $2-\chi -\mu \le 0$. Note that, in this case, there is no condition on $\sigma$ besides $\sigma>0$.\\[0.2cm] 
\noindent \emph{Counterexample in the range}
\begin{equation}\label{thirdrange_2}
0 \le \chi < 1, \qquad \mu < 2-\chi, \qquad \omega = 1-\chi.
\end{equation}
We define
\begin{equation}\label{def_sigma_dinuovo}
\sigma \doteq \frac{2-\chi-\mu}{1-\chi} >0.
\end{equation}
To get an unbounded solution of \eqref{cisiamo!!!} on some superlevel set we now consider $u(x)= h(r^2(x))$ with $h(t) = \exp\{(1+t)^{\sigma/2}\}$. Since $h' \ge 0$ and 
$$
\left|\frac{4r^2h''}{1+4r^2(h')^2}\right| \le C_0 \qquad \text{on } \, \R^+,
$$
for some explicit constant $C_0=C_0(\sigma,Q)$, then by \eqref{cumputa}
\begin{equation}\label{cumputa_2}
\begin{array}{l}
\disp \diver\left( \frac{\nabla u}{\sqrt{1+|\nabla u|^2}}\right) \ge \frac{1}{\sqrt{1+|\nabla u|^2}}\left[ 2(Q-1)h' -C_0\right] \\[0.5cm]
\quad \disp = \frac{1}{\sqrt{1+|\nabla u|^2}}\left[ \sigma(Q-1)h(1+r^2)^{\frac{\sigma-2}{2}} -C_0\right]  \ge \frac{C_1h}{\sqrt{1+|\nabla u|^2}}(1+r)^{\sigma-2},
\end{array}
\end{equation}
for a suitable $C_1>0$, and provided that we choose a high enough superlevel set $\{u>\gamma\}$. Since $\omega = 1-\chi$,
$$
u^\omega |\nabla u|^\chi = h^{1-\chi} \sigma^\chi \left[\frac{r}{\sqrt{1+r^2}}\right]^\chi (1+r^2)^{\frac{\chi(\sigma-1)}{2}}h^\chi \le C_3 h (1+r)^{\chi(\sigma-1)}. 
$$
Substituting into \eqref{cumputa_2} in place of the term $h$, we infer the existence of a constant $C>0$ such that
\begin{equation}\label{cumputa_3}
\begin{array}{lcl}
\disp \diver\left( \frac{\nabla u}{\sqrt{1+|\nabla u|^2}}\right) & \ge & \disp C(1+r)^{\sigma-2 - \chi(\sigma-1)}u^{\omega}\frac{|\nabla u|^\chi}{\sqrt{1+|\nabla u|^2}} \\[0.5cm]
& = & \disp C(1+r)^{-\mu}u^{\omega}\frac{|\nabla u|^\chi}{\sqrt{1+|\nabla u|^2}} \qquad \text{on } \, \Omega_\gamma,
\end{array}
\end{equation}
where the last equality follows from our definition of $\sigma$ in \eqref{def_sigma_dinuovo}. Therefore, $u$ is the desired counterexample for \eqref{thirdrange_2}. Observe that the need for an exponential growing $u$ is supported by Theorem 12 $(i)$ in \cite{farinaserrin}: indeed, applying their result to the mean curvature operator, with $p=2$ and $f$ satisfying $tf(t) \ge C|t|^{\omega+1}$ on $\R$, in the range \eqref{thirdrange_2} any solution with polynomial growth of 
\begin{equation}
\diver\left( \frac{\nabla u}{\sqrt{1+|\nabla u|^2}}\right) = C(1+r)^{-\mu}f(u)|\nabla u|^\chi \qquad \text{on } \, \R^n
\end{equation}
must be identically zero.\\
The bounds in \eqref{def_sigma_dinuovo} do not cover the case $\chi=1$, and we now conclude by commenting on the following
\begin{question}
What about possible counterexamples in the range
$$
\chi=1, \qquad \mu<2-\chi =1, \qquad \omega= 1-\chi =0?
$$
\end{question}
There is some evidence that, in the above range, Theorem \ref{teo_main} might still hold. Indeed, if we try to produce unbounded solutions of 
\begin{equation}\label{climax}
\diver\left( \frac{\nabla u}{\sqrt{1+|\nabla u|^2}}\right) \ge K(1+r)^{-\mu}\frac{|\nabla u|}{\sqrt{1+|\nabla u|^2}} \qquad \text{on } \, \Omega_\gamma
\end{equation}
of the type $u = h(r^2)$ with $h' \ge 0$, then using \eqref{cumputa} we would get
\begin{equation}\label{cumputa_final}
\disp \diver\left( \frac{\nabla u}{\sqrt{1+|\nabla u|^2}}\right) \left\{ \begin{array}{l} \disp \ge \frac{1}{\sqrt{1+|\nabla u|^2}} \left[ 2(Q-1)h' + \frac{4r^2h''}{1+4r^2(h')^2}\right] \\[0.5cm]
\disp \le \frac{1}{\sqrt{1+|\nabla u|^2}} \left[ 2Qh' + \frac{4r^2h''}{1+4r^2(h')^2}\right].
\end{array}\right.
\end{equation}
Being $|\nabla u| = 2rh'$, if the quotient with $h''$ in square brackets in \eqref{cumputa_final} is at most of the order of $h'$ as $r \ra +\infty$, then we deduce
$$
\diver\left( \frac{\nabla u}{\sqrt{1+|\nabla u|^2}}\right) \le \frac{Ch'}{\sqrt{1+|\nabla u|^2}} \le \frac{C}{r}\frac{|\nabla u|}{\sqrt{1+|\nabla u|^2}}.
$$
In this case, since $\mu<1$, $u$ can never solve \eqref{climax}. Computation \eqref{cumputa_final} suggests to search for $h$ satisfying
$$
\frac{4r^2h''(r^2)}{1+4r^2\big(h'(r^2)\big)^2} \ge \frac{C}{(1+r)^\mu} |\nabla u| \approx C(1+r)^{1-\mu}h'(r^2) \qquad \text{for large } \, r.
$$
However, it can be easily seen that the differential inequality
$$
\frac{4th''(t)}{1+4t\big(h'(t)\big)^2} \ge (1+t)^{\frac{1-\mu}{2}}h'(t) 
$$
possesses no increasing solutions defined on $[T, +\infty)$.

\section{Appendix}
In this Appendix we give a proof of the pasting lemma for locally Lipschitz solutions of inequalities of the type
\begin{equation}\label{eqB_2}
\Delta^\varphi_\GG u \ge B(x,u,\nabla_0u)
\end{equation}
(as $T\GG$ is parallelizable, we identify it with $\GG \times \R^n$, $n$ being the topological dimension of $\GG$). We assume that $\varphi$ satisfies
\begin{equation}\label{ipo_varphi_app}
\varphi \in C^0(\R^+_0), \qquad \varphi \ge 0 \quad \text{on } \, \R^+, 
\end{equation} 
and that $B : \GG \times \R \times \R^n$ enjoys the following properties:
\begin{equation}\label{ipo_B_app}
\begin{array}{l}
B \in L^\infty_\loc(\GG \times \R \times \R^n); \\[0.2cm] 
\text{for a.e. $x \in \GG$, $B(x, \cdot,\cdot)$ is continuous on  $\R\times \R^n$.}
\end{array}
\end{equation}
We recall that $v \in \lip_\loc(\Omega)$ is a solution of \eqref{eqB_2} if and only if 
		\begin{equation}\label{wB}
		\int_{\Omega}\left\{\frac{\varphi(\left|\nabla_0v\right|)}{\left|\nabla_0v\right|}\langle \nabla_0v,\nabla_0\phi \rangle +B(x,v,\nabla_0v)\phi\right\}\leq 0\,,
		\end{equation}
	for each $\phi\in C_0^{\infty}(\Omega)$, $\phi\geq0$ in $\Omega$. By approximation, the class of test functions for \eqref{wB} can be enlarged to $\phi \in \lip_c(\GG)$.

	\begin{lemma}\label{lemle}
	Let $\varphi,B$ satisfy \eqref{ipo_varphi_app}, \eqref{ipo_B_app}, fix an open set $\Omega \subset \GG$, and let $u_1$, $u_2\in \lip_\loc(\Omega)$ be solutions of
	\begin{equation}\label{eqB}
	\Delta_{\GG}^\varphi v \ge B(x,v,\nabla_0 v) \quad\hbox{on $\Omega$}\,.
	\end{equation}
	If none of $u_1$, $u_2$ is constant, assume further that
	\begin{equation}\label{ipo_addicional}
	\varphi(t) \quad \text{is non-decreasing on $\R^+$.}
	\end{equation}
	Then $u\doteq \max\left\{u_1, u_2\right\}\in \lip_\loc(\Omega)$ also solves \eqref{eqB} on $\Omega$.
	\end{lemma}

\begin{remark}
\emph{Observe that no monotonicity is required on $B(x,t,X)$ in the variable $t$, nor there are sign assumptions on $u_1,u_2$.
}
\end{remark}	
	
		\begin{proof}
		First of all we recall that $u=\frac{1}{2}\left\{\left(u_1+u_2\right)+\left|u_1-u_2\right|\right\}$. Thus, by Stampacchia's theorem (see Lemma 7.7 in \cite{GT}), 
			\begin{equation}
			\nabla_0 u=
				\begin{cases}
				\nabla_0 u_1 & \hbox{on $\Omega_1:=\left\{x\in\Omega : u_1>u_2\right\}$}\\
				\nabla_0 u_1=\nabla_0 u_2 & \hbox{on $\Omega_0:=\left\{x\in\Omega : u_1=u_2\right\}$}\\
				\nabla_0 u_2 & \hbox{on $\Omega_2:=\left\{x\in\Omega : u_1<u_2\right\}$}\,,
				\end{cases}
			\end{equation}
		from which it follows
			\begin{equation}
			B(x,u,\nabla_0 u)=B(x,u_1,\nabla_0 u_1)\chi_{\Omega\backslash \Omega_2}+B(x,u_2,\nabla_0 u_2)\chi_{\Omega_2}\,.
			\end{equation}
		To prove that $u$ is a solution of \eqref{eqB} we proceed as in \cite{Le}: consider 
			\begin{equation}								\begin{cases}
			\gamma:\R\rightarrow [0,1] \,,\quad\gamma\in C^{\infty}(\R)\\
			\gamma^{\prime}\geq 0\quad\hbox{on $\R$}\\
			\gamma(t)=0\,\, \hbox{if $t\leq0$},\quad \gamma(t)=1\,\, \hbox{if $t\geq1$}\,
			\end{cases}
			\end{equation}
	and, for $n\in\N$ and $t\in\R$, define $\gamma_n(t) \doteq \gamma(nt)$. Note that $\gamma_n(t)\nearrow\chi_{\R^+}(t)$ as $n\rightarrow+\infty$.	Consider a test function $0 \le \phi\in C^{\infty}_0(\Omega)$, and define $\phi_1, \phi_2 \in \lip_c(U)$ as being $\phi_1=\left[1-\gamma_n(u_2-u_1)\right]\phi$ and $\phi_2=\left[\gamma_n(u_2-u_1)\right]\phi$. By approximation, $\phi_1,\phi_2$ are admissible test functions, and clearly $\gamma_n(u_2-u_1) \ra \chi_{\Omega_2}$ pointwise. 
	Since both the $u_i$ are solutions of \eqref{eqB}, from \eqref{wB} we deduce
		\beq
		\int_{\Omega}\left\{\frac{\varphi(\left|\nabla_0u_i\right|)}{\left|\nabla_0u_i\right|}\langle \nabla_0u_i,\nabla_0\phi_i \rangle +B(x,u_i,\nabla_0u_i)\phi_i\right\}\leq 0\,,
		\eeq
	thus, adding the two inequalities, computing $\nabla_0\phi_i$ and rearranging, we get
		\beq\label{int0}
		\begin{aligned}
		0\geq & \int_{\Omega}\gamma_n^{\prime}\phi\langle \left[\frac{\varphi(\left|\nabla_0u_2\right|)}{\left|\nabla_0u_2\right|}\nabla_0u_2-\frac{\varphi(\left|\nabla_0u_1\right|)}{\left|\nabla_0u_1\right|}\nabla_0u_1\right],\nabla_0(u_2-u_1)\rangle \\
		& + \int_{\Omega}\gamma_n\langle\left[\frac{\varphi(\left|\nabla_0u_2\right|)}{\left|\nabla_0u_2\right|}\nabla_0u_2-\frac{\varphi(\left|\nabla_0u_1\right|)}{\left|\nabla_0u_1\right|}\nabla_0u_1\right],\nabla_0\phi \rangle\\
		& + \int_{\Omega}\left[B(x,u_2,\nabla_0u_2)-B(x,u_1,\nabla_0u_1)\right]\gamma_n\phi\\
		& + \int_{\Omega}\left\{\frac{\varphi(\left|\nabla_0u_1\right|)}{\left|\nabla_0u_1\right|}\langle\nabla_0u_1,\nabla_0\phi \rangle+B(x,u_1,\nabla_0u_1)\phi\right\}\,.
		\end{aligned}
		\eeq
	Using Cauchy-Schwarz inequality and $\gamma_n' \ge 0$, the first integral satisfy
		\beq\label{int1}
		\begin{aligned}
		& \int_{\Omega}\gamma_n^{\prime}\phi\langle \left[\frac{\varphi(\left|\nabla_0u_2\right|)}{\left|\nabla_0u_2\right|}\nabla_0u_2-\frac{\varphi(\left|\nabla_0u_1\right|)}{\left|\nabla_0u_1\right|}\nabla_0u_1\right],\nabla_0\left(u_2-u_1\right) \rangle\\
		& \quad\geq
		\int_{\Omega}\gamma_n^{\prime}\phi\left[\varphi(\left|\nabla_0u_2\right|)-\varphi(\left|\nabla_0u_1\right|)\right]\left(\left|\nabla_0u_2\right|-\left|\nabla_0u_1\right|\right).
        \end{aligned}
		\eeq
    The product 
    $$
    \left[\varphi(\left|\nabla_0u_2\right|)-\varphi(\left|\nabla_0u_1\right|)\right]\left(\left|\nabla_0u_2\right|-\left|\nabla_0u_1\right|\right)
    $$ 
    is non-negative because of \eqref{ipo_addicional} (if $u_1,u_2$ are both non-constant) or \eqref{ipo_varphi_app} (if one of them is constant). Hence, the first term in the right-hand side of \eqref{int0} is non-negative and can be thrown away. Letting then $n \ra +\infty$ in \eqref{int0} and using Lebesgue convergence theorem, we deduce  
		\beq
		\begin{aligned}
		0 & \geq \int_{\Omega_2}\langle \left[\frac{\varphi(\left|\nabla_0u_2\right|)}{\left|\nabla_0u_2\right|}\nabla_0u_2-\frac{\varphi(\left|\nabla_0u_1\right|)}{\left|\nabla_0u_1\right|}\nabla_0u_1\right],\nabla_0\phi \rangle\\
		&\quad\quad+\int_{\Omega_2}\left[B(x,u_2,\nabla_0u_2)-B(x,u_1,\nabla_0u_1)\right]\phi\\
		&\quad\quad+\int_{\Omega}\left\{\frac{\varphi(\left|\nabla_0u_1\right|)}{\left|\nabla_0u_1\right|}\langle\nabla_0u_1,\nabla_0\phi \rangle +B(x,u_1,\nabla_0u_1)\phi\right\}\\
		&= \int_{\Omega_2}\left\{\frac{\varphi(\left|\nabla_0u_2\right|)}{\left|\nabla_0u_2\right|}\langle\nabla_0u_2,\nabla_0\phi \rangle+B(x,u_2,\nabla_0u_2)\phi\right\}\\
		&\quad\quad+\int_{\Omega\setminus\Omega_2}\left\{\frac{\varphi(\left|\nabla_0u_1\right|)}{\left|\nabla_0u_1\right|}\langle \nabla_0u_1,\nabla_0\phi \rangle +B(x,u_1,\nabla_0u_1)\phi\right\}\\
		& = \int_{\Omega}\left\{\frac{\varphi(\left|\nabla_0u\right|)}{\left|\nabla_0u\right|}\langle \nabla_0u,\nabla_0\phi\rangle +B(x,u,\nabla_0u)\phi\right\}\,.
		\end{aligned}
		\eeq
	Since $\phi$ is an arbitrary test function, this concludes the proof of the lemma.
		\end{proof}

The pasting lemma below is a useful refinement of the above result.
		
	\begin{proposition}[The pasting lemma]\label{lem_pasting}
	Let $\Omega\subset\Omega^{\prime}\subset\GG$ be open domains, and suppose that $\varphi$ satisfies \eqref{ipo_varphi}. Let $u\in \lip_\loc(\Omega)$ and $u^{\prime}\in \lip_\loc(\Omega^{\prime})$ be such that
		\beq
		\begin{cases}
		\Delta_{\GG}^\varphi u \geq B(x,u,\nabla_0 u) &\hbox{on $\Omega$}\\
		\Delta_{\GG}^\varphi u^{\prime} \geq B(x,u^{\prime},\nabla_0 u^{\prime}) &\hbox{on $\Omega^{\prime}$}\\
		u^{\prime}(x)\geq u(x)&\hbox{on $\Omega^\prime \cap \partial\Omega$}\,.
		\end{cases}
		\eeq
	Then, the function $v\in \lip_\loc(\Omega^{\prime})$ defined as
		\beq
		v=
			\begin{cases}
			\max\left\{u,u^{\prime}\right\}&\hbox{on $\Omega$}\\
			u^{\prime}& \hbox{on $\Omega^{\prime}\setminus\Omega$}\,,
			\end{cases}
		\eeq
	satisfies
		\beq
		\Delta_{\GG}^\varphi v \geq B(x,v,\nabla_0 v) \quad\quad\hbox{on $\Omega^{\prime}$}\,.
		\eeq
	\end{proposition}
		\begin{proof}
		We first suppose that $u^{\prime}(x)>u(x)$ on $\Omega' \cap \partial\Omega$. By continuity of $u$ and $u^{\prime}$, $u^{\prime}(x)>u(x)$ still holds on some open set $V \subset \Omega'$ containing $\Omega' \cap \partial \Omega$. This implies that $v=u^{\prime}$ on $U \doteq (\Omega'\backslash \Omega) \cup V$ and thus
			\beq
			\Delta_{\GG}^\varphi v \geq B(x,v,\nabla_0 v) \quad\quad\hbox{on $U$}\,.
			\eeq 
		From Lemma \ref{lemle} it follows that
			\beq
			\Delta_{\GG}^\varphi v \geq B(x,v,\nabla_0 v) \quad\quad\hbox{on $\Omega$}\,.
			\eeq 
		Since $\Omega^{\prime}=\Omega \cup U$ the proposition is proved.\\
		To deal with the general case, for $\eps>0$ we set $u_{\eps}(x) \doteq u(x)-\eps$, and we define
		\begin{equation}
		B_\eps(x,t,X) = \min\left\{ B(x,t,X), B(x,t-\eps,X)\right\}.
		\end{equation}
		Then, clearly 
		\beq
		\begin{cases}
		\Delta_{\GG}^\varphi u_\eps \geq B_\eps(x,u_\eps,\nabla_0 u_\eps) &\hbox{on $\Omega$}\\
		\Delta_{\GG}^\varphi u^{\prime} \geq B_\eps(x,u^{\prime},\nabla_0 u^{\prime}) &\hbox{on $\Omega^{\prime}$}\\
		u^{\prime}(x)> u_\eps(x)&\hbox{on $\Omega^\prime \cap \partial\Omega$}\,.
		\end{cases}
		\eeq
		By the first part of the proof, the function
		\beq
		v_{\eps}=
			\begin{cases}
			\max\left\{u_{\eps},u^{\prime}\right\}&\hbox{on $\Omega$}\\
			u^{\prime}& \hbox{on $\Omega^{\prime}\setminus\Omega$}\,,
			\end{cases}
		\eeq
		solves
			\beq
			\Delta_{\GG}^\varphi v_{\eps} \geq B_\eps(x,v_{\eps},\nabla_0 v_{\eps}) \quad\quad\hbox{on $\Omega^{\prime}$}\,.
			\eeq
		Thus, for any test function $\phi\in C^{\infty}_0(\Omega^{\prime})$, $\phi\ge 0$ we have that
			\beq
			\begin{aligned}\label{finaleequaz}
			0&\geq\int_{\Omega^{\prime}}\left\{\frac{\varphi(\left|\nabla_0v_{\eps}\right|)}{\left|\nabla_0v_{\eps}\right|}\langle \nabla_0v_{\eps},\nabla_0\phi \rangle +B_\eps(x,v_{\eps},\nabla_0v_{\eps})\phi\right\}\\
			\end{aligned}
			\eeq
		To conclude, we need to check that the terms in \eqref{finaleequaz} appropriately converge as $\eps \ra 0$. By the definition of $B_\eps$, $v_\eps,v$, the continuity of $B(x, \cdot, \cdot)$ and Lebesgue theorem,
			 \begin{equation}
		     \begin{array}{l}
		     \disp \int \phi\left| B_\eps(x,v_\eps, \nabla_0 v_\eps) - B(x,v,\nabla_0 v)\right| \\[0.5cm]
		     \quad \disp \le \int \phi\Big[\left| B_\eps(x,v_\eps, \nabla_0 v_\eps) - B(x, v_\eps, \nabla_0 v_\eps)\right| + \left| B(x,v_\eps, \nabla_0 v_\eps) - B(x, v, \nabla_0 v)\right|\Big] \\[0.5cm]
		     \quad \disp \le \int \phi\left| B(x,v_\eps- \eps, \nabla_0 v_\eps) - B(x, v_\eps, \nabla_0 v_\eps)\right| \\[0.5cm]
		     \qquad + \disp \int_{\Omega \cap \{u \in [u^\prime,u'+\eps]\}} \phi \left| B(x,u^\prime, \nabla_0 u^\prime) - B(x, u, \nabla_0 u)\right| \\[0.5cm]
		     \quad \disp \ra \int_{\Omega \cap \{u'= u\}} \phi \left| B(x,u^\prime, \nabla_0 u^\prime) - B(x, u, \nabla_0 u)\right| = 0,
		     \end{array}
             \end{equation}		
			as $\eps \ra 0$, where the last equality follows from Stampacchia's theorem. In a similar way,
              \begin{equation}
		     \begin{array}{l}
		     \disp \int \left|\frac{\varphi(\left|\nabla_0v_{\eps}\right|)}{\left|\nabla_0v_{\eps}\right|}\nabla_0v_{\eps}-\frac{\varphi(\left|\nabla_0v\right|)}{\left|\nabla_0v\right|} \nabla_0v\right| |\nabla_0\phi| \\[0.5cm]
		     \qquad \disp = \int_{\Omega \cap \{u \in [u^\prime, u^\prime +\eps]\}} \left|\frac{\varphi(\left|\nabla_0u^\prime\right|)}{\left|\nabla_0u^\prime\right|}\nabla_0u^\prime-\frac{\varphi(\left|\nabla_0u\right|)}{\left|\nabla_0u\right|}\nabla_0u\right| |\nabla_0\phi| \\[0.5cm]
		     \qquad \ra \disp \int_{\Omega \cap \{u^\prime =u\}} \left|\frac{\varphi(\left|\nabla_0u^\prime\right|)}{\left|\nabla_0u^\prime\right|}\nabla_0u^\prime-\frac{\varphi(\left|\nabla_0u\right|)}{\left|\nabla_0u\right|}\nabla_0u\right| |\nabla_0\phi| = 0.
     	     \end{array}
             \end{equation}
             Letting $\eps \ra 0$ in \eqref{finaleequaz} we thus get
             \beq
			\begin{aligned}
			0&\geq\int_{\Omega^{\prime}}\left\{\frac{\varphi(\left|\nabla_0v\right|)}{\left|\nabla_0v\right|}\langle \nabla_0v,\nabla_0\phi \rangle +B(x,v,\nabla_0v)\phi\right\},
			\end{aligned}
			\eeq
			which was to be proved.
		\end{proof}
\section*{Acknowledgements}
The second author is supported by the grant PRONEX - N\'ucleo de An\'alise Geom\'etrica e Aplicac\~oes. Proceso n. PR2-0054-00009.01.00/11.\\
The authors would like to thank the referee for his/her careful reading, and for valuable suggestions that helped to improve our results.
	
\bibliographystyle{amsplain}
\bibliography{bibliohardy_heisenberg}

\end{document}